\documentclass[12pt,reqno]{amsart}
\usepackage{amsmath}
\usepackage{latexsym,amssymb,eucal}
\usepackage{mathrsfs}
\usepackage{enumerate}
\usepackage{bm}
\allowdisplaybreaks[1]

\newtheorem{thm}{Theorem}[section]
\newtheorem{lem}[thm]{Lemma}
\newtheorem{cor}[thm]{Corollary}
\newtheorem{prop}[thm]{Proposition}
\newtheorem{prob}[thm]{Problem}

\newtheorem*{clm}{Claim}

\newtheorem{thme}[]{Theorem}

\theoremstyle{definition}

\newtheorem{defn}[thm]{Definition}
\newtheorem{rem}[thm]{Remark}
\newtheorem{exam}[thm]{Example}

\numberwithin{equation}{section}

\DeclareMathOperator{\Ad}{Ad}
\DeclareMathOperator{\Aut}{Aut}
\DeclareMathOperator{\ch}{ch}
\DeclareMathOperator{\diag}{diag}
\DeclareMathOperator{\End}{End}
\DeclareMathOperator{\ev}{ev}
\DeclareMathOperator{\Hom}{Hom}

\DeclareMathOperator{\id}{id}
\DeclareMathOperator{\Ima}{Im}
\DeclareMathOperator{\Ind}{Ind}
\DeclareMathOperator{\Int}{Int}
\DeclareMathOperator{\oInt}{\overline{\Int}}
\DeclareMathOperator{\Irr}{Irr}
\DeclareMathOperator{\mo}{mod}
\DeclareMathOperator{\Rep}{Rep}
\DeclareMathOperator{\Sect}{Sect}
\DeclareMathOperator{\Sp}{Sp}
\DeclareMathOperator{\spa}{span}
\DeclareMathOperator{\tr}{tr}
\DeclareMathOperator{\Tr}{Tr}
\DeclareMathOperator{\Wt}{Wt}

\def\C{\mathbb{C}}
\def\N{\mathbb{N}}
\def\Q{\mathbb{Q}}
\def\R{\mathbb{R}}
\def\Z{\mathbb{Z}}

\def\bG{\mathbb{G}}
\def\bH{\mathbb{H}}
\def\bK{\mathbb{K}}

\def\cA{\mathcal{A}}

\def\cL{\mathcal{L}}
\def\cM{\mathcal{M}}
\def\cN{\mathcal{N}}
\def\cP{\mathcal{P}}
\def\cQ{\mathcal{Q}}
\def\cR{\mathcal{R}}

\def\sG{\mathscr{G}}
\def\sH{\mathscr{H}}

\def\al{\alpha}
\def\bal{{\overline{\al}}}

\def\be{\beta}

\def\ga{\gamma}
\def\de{\delta}
\def\ka{\kappa}
\def\la{\lambda}

\def\vep{\varepsilon}
\def\ph{{\phi}}
\def\ps{{\psi}}

\def\vph{\varphi}
\def\ovph{{\overline{\vph}}}

\def\om{\omega}
\def\si{\sigma}
\def\ta{\tau}
\def\th{\theta}

\def\De{\Delta}
\def\Ga{\Gamma}
\def\La{\Lambda}
\def\Om{\Omega}

\def\col{\colon}
\def\nin{\notin}
\def\ra{\rightarrow}
\def\subs{\subset}
\def\ovl{\overline}
\def\oti{\otimes}
\def\rti{\rtimes}

\def\CG{C(G_q)}
\def\lG{L^\infty(G_q)}
\def\CT{C(T)}
\def\CTG{C(T\backslash G_q)}
\def\lTG{L^\infty(T\backslash G_q)}
\def\RG{R(\bG)}
\def\Uqg{U_q(\mathfrak{g})}

\title{Product type actions of $G_q$}

\author[R. Tomatsu]{Reiji Tomatsu}
\address{
Department of Mathematics, Hokkaido University,
Hokkaido\\ \indent \mbox{060-0810},
JAPAN}
\email{tomatsu@math.sci.hokudai.ac.jp}

\begin{document}
\maketitle

\begin{abstract}
We will study a faithful product type action
of $G_q$, the $q$-deformation
of a connected semisimple compact Lie group $G$,
and prove that such an action
is induced from a minimal action
of the maximal torus $T$ of $G_q$.
This enables us to classify product type actions
of $SU_q(2)$ up to conjugacy.
We also compute
the intrinsic group of $G_{q,\Omega}$,
the 2-cocycle deformation of $G_q$
that is naturally associated with
the quantum flag manifold $L^\infty(T\backslash G_q)$.
\end{abstract}

\section{Introduction}
In this paper,
we will study a product type action
of a $q$-deformed compact quantum group.
Theory of a quantum group was initiated by
Drinfel'd and Jimbo \cite{D,J}.
They have introduced the quantum group
$U_q(\mathfrak{g})$,
the $q$-deformation of
the enveloping algebra of a Kac--Moody Lie algebra
$\mathfrak{g}$.
In the operator algebraic approach,
Woronowicz has defined $SU_q(N)$
and introduced the concept of
a compact quantum group \cite{W,W-pseudo,W-SU(N),W-cpct}
by deforming a function algebra.
We can construct
the $q$-deformed compact quantum group $G_q$
for a connected semisimple compact Lie group $G$
using $U_q(\mathfrak{g})$
for a finite dimensional simple $\mathfrak{g}$
(see \cite{KS,FRT}).

Let us consider a product type action of $G_q$
on a uniformly hyperfinite C$^*$-algebra.
Such an action
has been studied by Konishi--Nagisa--Watatani \cite{KNW}.
They have shown that
the fixed point algebra of the product type action
of $SU_q(2)$ with respect to the spin-$1/2$ irreducible
representation
is generated by certain Jones projections.
In particular, if we take its weak closure,
then the product type action is never minimal.

In \cite{Iz-Poisson},
Izumi has elucidated this interesting phenomenon
by introducing the concept of a (non-commutative) Poisson boundary.
Namely,
he has constructed a von Neumann algebra from
a random walk on the dual discrete quantum group
of a compact quantum group,
and shown that is isomorphic to the relative commutant
of a fixed point algebra inside an infinite tensor product
factor of matrix algebras.
Since this pioneering work,
the program of realization
of Poisson boundaries has been carried out
in several papers \cite{Iz-Poisson,INT,T-Poisson,VV0,VV}.
In particular, it is known that
the Poisson boundary of $G_q$ is isomorphic to
the quantum flag manifold $T\backslash G_q$
\cite{Iz-Poisson,INT,T-Poisson}.

On the center of a Poisson boundary
as a von Neumann algebra,
it has been conjectured in \cite{T-Banach}
that
the center could coincide with
the classical part of the Poisson boundary,
that is,
the center could come exactly from the random walk
on the von Neumann algebraic
center of the dual discrete quantum group.
Indeed, it is well-known for experts
that this is the case
to $SU_q(2)$.
Also for universal quantum groups
$A_o(F)$ and $A_u(F)$,
the conjecture has been affirmatively solved
\cite{TV,VV0,VV}.
We will show the following result
which states that the conjecture also holds
for every $G_q$ (Theorem \ref{thm:factor}).

\begin{thme}
\label{thm:intro-factor}
The von Neumann algebra
$L^\infty(T\backslash G_q)$
is a factor of type I.
\end{thme}

Let us again
consider a product type action of $G_q$
on a factor $\cM$.
From Izumi's result and Theorem \ref{thm:intro-factor},
it turns out that the relative commutant
$\cQ:=(\cM^{G_q})'\cap\cM$ is the infinite dimensional
type I factor, where $\cM^{G_q}$ denotes
the fixed point algebra.
Thus we obtain the tensor product splitting
$\cM\cong\cR\oti \cQ$, where $\cR:=\cQ'\cap\cM$.
It is then shown that
the inclusion $\cM^{G_q}\subs \cR$ is irreducible
and of depth 2 (Lemma \ref{lem:depth2}).
Hence it arises from a minimal action
of a unique compact quantum group on $\cR$.
Actually, we will show that the compact quantum group
is nothing but the maximal torus $T$
(Theorem \ref{thm:depth2-torus}).
Hence we have a $T$-equivariant copy of $L^\infty(T)$
inside $\cR$.
Then it is natural to ask whether
this copy and $\cQ$ generate a von Neumann algebra
that is $G_q$-isomorphic to $\lG$.

To solve this question,
we need to show the triviality of
the $G_q$-equivariant automorphism on $\lTG$.
In \cite{S-repn},
Soibel'man has classified irreducible representations
of $\CG$ as a C$^*$-algebra.
Namely, it has become clear that
irreducible representations of $\CG$
are parametrized by
the maximal torus $T$
and
the Weyl group $W$ of $\mathfrak{g}$
as $\{\pi_{t,w}\}_{t\in T,w\in W}$.
Then Dijkhuizen--Stokman's result \cite[Theorem 5.9]{DS}
(Theorem \ref{thm:class-irr})
states that
any irreducible representation of $\CTG$ actually comes from
that of $\CG$.
This, in particular, implies that
the counit gives a unique character on $\CTG$,
and we obtain the triviality of the $G_q$-equivariant
automorphism group of $\CTG$
(Corollary \ref{cor:equiv}).

By using these results,
we study an $\lTG$-valued invariant cocycle.
Actually, it is shown that any such cocycle is a coboundary
with a unique solution in $Z(\lG)$
up to a scalar multiple
(Theorem \ref{thm:descript-coc}).
As an application,
we can show the following main result of this paper
(Theorem \ref{thm:induction}).

\begin{thme}
A faithful product type action
of $G_q$ is induced from a minimal
action of $T$ on a type III factor.
Moreover, such a minimal action is unique.
\end{thme}

Another application of Theorem \ref{thm:intro-factor}
concerns theory of 2-cocycle deformation
of locally compact quantum groups.
In \cite{DC},
De Commer has shown that
a 2-cocycle twisted von Neumann bi-algebra
of a locally compact quantum group
again has a locally compact quantum group structure.
In our setting,
we will encounter with a 2-cocycle $\Om$
that is canonically
associated with an irreducible projective unitary
representation of $G_q$ coming from $\lTG$.
We can determine the intrinsic group of $G_{q,\Om}$,
the deformation of $G_q$ by $\Om$
(Theorem \ref{thm:intrinsic}) as follows.

\begin{thme}
The intrinsic group of $G_{q,\Om}$
is isomorphic to $\widehat{T}$.
\end{thme}

When $G_q=SU_q(2)$,
it has been proved by De Commer
that
$G_{q,\Om}$ is isomorphic to $\widetilde{E}_q(2)$,
Woronowicz's quantum $E(2)$ group
\cite{DC-E2}.
The above theorem generalizes
a partial result of \cite{W-E2}.

This paper is organized as follows.

In Section 2,
we will give a brief summary of theory of a compact quantum group
and a $q$-deformed Lie group $G_q$.

In Section 3,
We will show the factoriality of $\lTG$,
and
present an alternative proof
of Dijkhuizen--Stokman's classification
result (Theorem \ref{thm:class-irr}).
We especially emphasize that
we use the von Neumann algebra $\lTG$
to classify all irreducible representations
of the C$^*$-algebra $\CTG$.
As an application,
we will compute a density operator of the Haar state
(Theorem \ref{thm:Haar-state-density})
and derive the well-known
quantum Weyl dimension formula
(Proposition \ref{prop:quantum-dim}).

In Section 4,
the notion of an invariant cocycle is introduced.
We will prove that all invariant cocycles evaluated in $\lTG$
come from the canonical generators of $Z(\lG)$.
As an application,
we compute the intrinsic group of a 2-cocycle deformation
$G_{q,\Om}$.

In Section 5,
we will discuss a product type action.
First, we will deduce from sector theory
that
the canonical inclusion of factors stated before
corresponds to a minimal action of the maximal torus $T$.
Next, by using invariant cocycles
and the triviality of $G_q$-equivariant automorphism
of $\lTG$,
we will show that a faithful product type action
is actually induced from a minimal action of $T$.
Then the classification of product type actions is studied.
Especially, we will present a complete classification of product type actions
of $SU_q(2)$.
Uncountably many non-product type and mutually non-cocycle conjugate
actions of $SU_q(2)$
on the injective type III$_1$ factors
are also constructed.

In the last section,
we will pose a problem concerning the main results
in a more general situation.

\vspace{10pt}
\noindent
{\bf Acknowledgements.}
The author is grateful
to Masaki Izumi for various advice.
He also would like to thank
Noriyuki Abe and Kenny De Commer
for valuable comments on this paper.
This work is supported in part by
JSPS KAKENHI Grant Number 24740095.

\section{Preliminary}
\subsection{Notations and terminology}
In this paper,
$\Z_+$ denotes the set of non-negative integers,
that is,
$\Z_+=\{0,1,\dots\}$.

The tensor symbol
$\otimes$ denotes the minimal tensor product
for C$^*$-algebras
and the von Neumann algebra tensor product
for von Neumann algebras.

We denote by $\spa S$ and $\ovl{\spa}^{\rm w}S$,
the linear span of a set $S$
and
the weak closure of $\spa S$,
respectively.

For a von Neumann algebra $\cM$,
we will denote by $Z(\cM)$ its center.
By $\End(\cM)$,
we will denote
the set of normal endomorphisms on $\cM$.
For $\rho,\si\in\End(\cM)$,
$(\rho,\si)$ denotes the set of
intertwiners.
Namely, an element $a\in (\rho,\si)$ satisfies
$a\rho(x)=\si(x)a$ for all $x\in\cM$.
If $(\rho,\rho)=\C$,
then we will say that $\rho$ is irreducible.
Two endomorphisms $\rho,\si$ on $\cM$
are said to be equivalent
if there exists a unitary
$u\in\cM$ such that $\rho=\Ad u\circ\si$.
By $\Sect(\cM)$,
we denote the quotient space of
$\End(\cM)$.
The equivalence class of $\rho$
is denoted by $[\rho]$
which is called a \emph{sector}.
For sector theory,
reader's are referred to
\cite{Iz-fusion,L-ind1,L-ind2}.

Recall the notion of a Hilbert space
in a von Neumann algebra \cite{Rob}.
A weakly closed
linear space $\sH$ in a von Neumann algebra $\cM$
is called a \emph{Hilbert space in} $\cM$
if $W^*V\in\C$
for all $V,W\in\sH$.
Then $\sH$ is a Hilbert space
with the inner product $\langle V,W\rangle:=W^*V$.
The support of $\sH$,
which we denote by $s(\sH)$,
is the infimum of projections
$p\in\cM$ such that $pV=V$ for all $V\in\sH$.
If $\{V_i\}_{i\in I}$ is an orthonormal base
of $\sH$,
then we have $s(\sH)=\sum_{i\in I}V_iV_i^*$.

If $\rho,\si\in\End(\cM)$ and $\rho$ is irreducible,
then $(\rho,\si)$ is a Hilbert space in $\cM$
by the inner product $\langle V,W\rangle:=W^*V$
for $V,W\in (\rho,\si)$.

Let $\cN\subs\cM$ be an inclusion of
properly infinite von Neumann algebras.
Then $L^2(\cM)$ also has
the structure of the standard form for $\cN$.
Let $J_\cM$ and $J_\cN$ be
the modular conjugations of $\cM$
and $\cN$, respectively.
Then
$\ga_\cN^\cM(x):=J_\cN J_\cM x J_\cM J_\cN$,
$x\in \cM$,
is called the \emph{canonical endomorphism}
from $\cM$ into $\cN$.
It is known that
the sector
$[\ga_\cN^\cM]$ in $\Sect(\cN)$
does not depend on the choice
of the structure of the standard forms of $\cN$
and $\cM$.

\subsection{Compact quantum group}
We will quickly review theory of compact quantum groups
introduced by Woronowicz.
Our references are \cite{T-Poisson,W-cpct}.

\begin{defn}[Woronowicz]
We will say
that a pair $(A,\de)$ of a separable unital C$^*$-algebra $A$
and a faithful unital $*$-homomorphism
$\de\col A\ra A\oti A$
is a \emph{compact quantum group}
when the following conditions hold:
\begin{itemize}
\item
$\de$ is a coproduct,
that is,
$(\de\oti\id)\circ\de=(\id\oti\de)\circ\de$;
\item
$\de(A)(\C\oti A)$
and
$\de(A)(A\oti\C)$
are norm dense subspaces in $A\oti A$.
\end{itemize}
\end{defn}

If $\bG:=(A,\de)$ is a compact quantum group,
we write $C(\bG):=A$.
It is known that there exists a unique state $h$
called the \emph{Haar state}
such that
\[
(\id\oti h)(\de(x))=h(x)1=(h\oti\id)(\de(x))
\quad
\mbox{for all }
x\in C(\bG).
\]
We always assume
that $h$ is faithful in what follows.
Let $\{L^2(\bG),1_h\}$ be the GNS representation
with respect to $h$.
We will regard $C(\bG)$ as a C$^*$-subalgebra
of $B(L^2(\bG))$ from now on.
By $L^\infty(\bG)$,
we denote the weak closure of $C(\bG)$.

The \emph{multiplicative unitary} $V$
is a unitary
on $L^2(\bG)\oti L^2(\bG)$
such that
\[
V(x1_h\oti\xi)=\de(x)(1_h\oti\xi)
\quad
\mbox{for }
x\in C(\bG),
\
\xi\in L^2(\bG).
\]
Then $V$ satisfies
the pentagon equation
$V_{12}V_{13}V_{23}=V_{23}V_{12}$.
The coproduct $\de$
extends to the normal coproduct
$\de\col L^\infty(\bG)\ra L^\infty(\bG)\oti L^\infty(\bG)$
by
\[
\de(x)=V(x\oti1)V^*
\quad
\mbox{for }
x\in L^\infty(\bG).
\]
Note $V$ belongs to $B(L^2(\bG))\oti L^\infty(\bG)$.
The Haar state $h$ also extends to
a faithful normal invariant state on
$L^\infty(\bG)$
by putting
$h(x):=\langle x1_h,1_h\rangle$
for $x\in L^\infty(\bG)$.

Let $H$ be a Hilbert space.
A unitary $v\in B(H)\oti L^\infty(\bG)$
is called a \emph{unitary representation} on $H$
when $(\id\oti\de)(v)=v_{12}v_{13}$.
If $v=(v_{ij})_{i,j}$
is the matrix representation
with $v_{ij}\in L^\infty(\bG)$,
then we have $\de(v_{ij})=\sum_k v_{ik}\oti v_{kj}$.

Let $v$ and $w$ be unitary representations
on Hilbert spaces $H$ and $K$, respectively.
The intertwiner space $(v,w)$
is the set of bounded linear operators
$a\col H\ra K$
such that $(a\oti1)v=w(a\oti1)$.
If $(v,v)=\C1_H$,
then $v$ is said to be \emph{irreducible}.
If this is the case,
$H$ must be finite dimensional.

By $\Rep_f(\bG)$,
we denote the set of finite dimensional
unitary representations.
We set
\[
A(\bG)
:=\spa\{(\om\oti\id)(v)\mid \om\in B(H)_*,
\
v\in\Rep_f(\bG)
\}.
\]
Then $A(\bG)$ is a dense unital $*$-subalgebra
of $C(\bG)$.
We can define an anti-multiplicative
linear map $\ka$ on $A(\bG)$,
which is called the \emph{antipode},
such that
$(\id\oti\ka)(v)=v^*$
for $v\in\Rep_f(\bG)$.
The \emph{counit}
is the character
$\vep\col A(\bG)\ra\C$
such that $(\id\oti\vep)(v)=1$
for $v\in\Rep_f(\bG)$.
We only treat a co-amenable $\bG$ in this paper,
and $\vep$ extends to the character on
$C(\bG)$.
See \cite{BCT,BMT1,BMT2,T-ame} for details of amenability.

We will introduce the Woronowicz characters $\{f_z\}_{z\in\C}$.
They are multiplicative linear functionals on $A(\bG)$
uniquely determined by the following properties:
\begin{enumerate}
\item
$f_0=\vep$;
\item
For any $a\in A(\bG)$,
the function $\C\ni z\mapsto f_z(a)\in\C$
is entirely holomorphic;

\item
$(f_{z_1}\oti f_{z_2})\circ\de=f_{z_1+z_2}$
for all $z_1,z_2\in\C$;

\item
$f_z(\ka(a))=f_{-z}(a)$,
$f_z(a^*)=\ovl{f_{-\bar{z}}(a)}$
for all $z\in\C$ and $a\in A(\bG)$;

\item
$\ka^2=(f_1\oti\id\oti f_{-1})\circ\de^{(2)}$;

\item
$h(ab)=h(b (f_1\oti\id\oti f_1)(\de^{(2)}(a)))$
for all $a,b\in A(\bG)$,
\end{enumerate}
where $\de^{(2)}:=(\de\oti\id)\circ\de$.
In general,
for $k\in\N$, we let
$\de^{(k)}:=(\de^{(k-1)}\oti\id)\circ\de$.
Then
$\de^{(k+\ell)}=(\de^{(k)}\oti\de^{(\ell-1)})\circ\de$
for $k,\ell\in\N$.

The modular automorphism group
$\si^h$ is given by
\[
\si_t^h(x)=(f_{it}\oti \id\oti f_{it})(\de^{(2)}(x))
\quad
\mbox{for all }
t\in\R,\ x\in A(\bG).
\]
Define the \emph{scaling automorphism group} $\ta$ by
\[
\ta_t(x)=(f_{it}\oti \id\oti f_{-it})(\de^{(2)}(x))
\quad
\mbox{for all }
t\in\R,\ x\in A(\bG).
\]

Let $v\in B(H)\oti L^\infty(\bG)$ be a finite dimensional
unitary representation.
Then it is known that
$v$ in fact belongs to $B(H)\oti A(\bG)$.
We let $F_v:=(\id\oti f_1)(v)$,
which is non-singular and positive.
Then $F_v^z=(\id\oti f_z)(v)$ for all $z\in\C$.
It is known that $\Tr(F_v)=\Tr(F_v^{-1})$,
which is called the \emph{quantum dimension} of $v$,
and denoted by $\dim_q(v)$ or $\dim_q H$.

Let $\Irr(\bG)$ be the complete set of
unitary equivalence classes
of irreducible unitary representations.
For $s\in \Irr(\bG)$,
we fix a section $v(s)=(v(s)_{ij})_{i,j\in I_s}$.
Then we have the following orthogonal equalities:
for all $s,t\in\Irr(\bG)$,
$i,j\in I_s$ and $k,\ell\in I_t$,
\begin{equation}
\label{eq:Haar-ortho}
h(v(s)_{ij}v(t)_{k\ell}^*)
=
\dim_q(v(s))^{-1}(F_{v_s})_{\ell,j}\de_{s,t}\de_{i,k},
\end{equation}
\[
h(v(s)_{ij}^*v(t)_{k\ell})
=
\dim_q(v(s))^{-1}(F_{v_s}^{-1})_{k,i}\de_{s,t}\de_{j,\ell}.
\]

\subsection{Action}
Let $A$ be a unital C$^*$-algebra.
We will say that
a faithful unital $*$-homomorphism $\al\col A\ra A\oti C(\bG)$
is a \emph{(right) action} of $\bG$ on $A$
if $(\al\oti\id)\circ\al=(\id\oti\de)\circ\al$,
and $\al(A)(\C\oti C(\bG))$ is a dense subspace
of $A\oti C(\bG)$.
Similarly, we can define a left action.

By $A^\al$, we denote the fixed point algebra
$\{x\in A\mid \al(x)=x\oti1\}$.
If $A^\al=\C$,
then $\al$ is said to be \emph{ergodic}.

For a von Neumann algebra $\cM$,
a (right) action means a faithful normal unital
$*$-homomorphism
$\al\col\cM\ra\cM\oti L^\infty(\bG)$
satisfying $(\al\oti\id)\circ\al=(\id\oti\de)\circ\al$.
It is known that
the condition of the density like the above automatically holds.
The fixed point algebra $\cM^\al$ is similarly defined.

We will say that a state $\vph\in \cM_*$
is \emph{invariant}
when
$(\vph\oti\id)(\al(x))=\vph(x)1$ for all $x\in\cM$.

The crossed product is defined by
\[
\cM\rti_\al\bG
:=
\ovl{\spa}^{\rm w}
\{\al(\cM)(\C\oti\RG)\}
\subs
\cM\oti B(L^2(\bG)),
\]
where $\RG$ denotes the right quantum group algebra,
that is,
\[
\RG:=\ovl{\spa}^{\rm w}\{(\id\oti\om)(V)\mid \om\in L^\infty(\bG)_*\}.
\]

Let $e_1:=(\id\oti h)(V)\in\RG$.
Then $e_1$ is a minimal projection of $B(L^2(\bG))$,
and $(1\oti e_1)(\cM\rti_\al\bG)(1\oti e_1)=\cM^\al\oti \C e_1$.
Thus $\cM^\al$ is a corner of $\cM\rti_\al\bG$.
In particular,
if $\cM\rti_\al\bG$ is a factor,
then so is $\cM^\al$.

\subsection{Quantum subgroup}
\label{subsect:quantum subgrp}
Let $\bH$ and $\bG$ be compact quantum groups.
We will say that $\bH$ is a \emph{quantum subgroup} of $\bG$
when there exists a unital surjective $*$-homomorphism
$r_\bH\col C(\bG)\to C(\bH)$,
which we will call a \emph{restriction map},
such that
$\de_\bH\circ r_\bH=(r_\bH\oti r_\bH)\circ\de_\bG$,
where $\de_\bH$ and $\de_\bG$ denote
the coproducts of $\bH$ and $\bG$, respectively.

Then $\bH$ acts on $C(\bG)$ from the both sides.
Namely,
let $\ga_\bH^\ell:=(r_\bH\oti\id)\circ\de_\bG$
and $\ga_\bH^r:=(\id \oti r_\bH)\circ\de_\bG$.
Then they are left and right actions of $\bH$.
Moreover, they are commuting,
that is,
$(\id\oti\ga_\bH^r)\circ\ga_\bH^\ell
=(\ga_\bH^\ell\oti\id)\circ\ga_\bH^r$.

Let us introduce the function algebras
on the homogeneous spaces as follows:
\[
C(\bH\backslash\bG)
:=
\{x\in C(\bG)\mid \ga_\bH^\ell(x)=1\oti x\},
\]
\[
C(\bG/\bH)
:=
\{x\in C(\bG)\mid \ga_\bH^r(x)=x\oti1\}.
\]
Then the restrictions of $\de_\bG$
on $C(\bH\backslash\bG)$
and $C(\bG/\bH)$ yield actions from the right
and left, respectively.
The weak closures of $C(\bH\backslash\bG)$
and $C(\bG/\bH)$ are denoted by
$L^\infty(\bH\backslash\bG)$
and $L^\infty(\bG/\bH)$, respectively.

Let $\vep_\bH$ and $\vep_\bG$ be
the counits of $\bH$ and $\bG$,
respectively.
Then $\vep_\bH\circ r_\bH=\vep_\bG$.
So we will denote simply by $\vep$ the counits
of $\bH$ and $\bG$.

\begin{lem}
Let $\al$
be an action of $\bG$ on a von Neumann algebra $\cM$.
Then there uniquely exists a unital ${\rm C}^*$-subalgebra
$A$ of $\cM$
such that
\begin{itemize}
\item
$\al(A)\subs A\oti C(\bG)$;

\item
If a ${\rm C}^*$-subalgebra $B\subs\cM$
satisfies $\al(B)\subs B\oti C(\bG)$,
then $B\subs A$;

\item
$A$ is weakly dense in $\cM$.
\end{itemize}
\end{lem}
\begin{proof}
Let $A$ be a C$^*$-algebra generated by
all C$^*$-subalgebras
$B\subs\cM$
which satisfy $\al(B)\subs B\oti C(\bG)$.
Then it is clear that (1) and (2) hold.
We will check (3) as follows.

Let $s\in\Irr(\bG)$
and $H_s$ the corresponding
irreducible module.
Let $\Hom_\bG(H_s,\cM)$ be the set of
$\bG$-equivariant linear maps.
Then $\cM$ is weakly spanned by
$T(H_s)$ for $T\in \Hom_\bG(H_s,\cM)$
and $s\in\Irr(\bG)$.
It is clear that $T(H_s)\subs A$,
and we are done.
Note that $A$ is in fact
generated by such $T(H_s)$'s.
\end{proof}

Let $\al$ and $A$ be as above
and $\bH$ a quantum subgroup of $\bG$
with a restriction map $r_\bH$.
Then we can restrict $\al$ on $\bH$,
that is,
$\al_\bH:=(\id\oti r_\bH)\circ\al$
gives an action of $\bH$ on $A$.
It is clear that $\al_\bH$ preserves
any $\al$-invariant normal state on $\cM$.
Thus $\al_\bH$ extends to $\cM$
as the action of $\bH$.
We will call $\al_\bH$ the restriction of $\al$
by $\bH$.

\subsection{$U_q(\mathfrak{g})$}
We will review the definition of $U_q(\mathfrak{g})$
introduced by Drinfel'd and Jimbo \cite{D,J},
and the highest weight theory.
Our references are \cite{CP,Jos,Kac,Kl-Sch,KS}.
Let $A=(a_{ij})_{i,j\in I}$ be
an irreducible Cartan matrix of finite type
($I:=\{1,\dots,n\}$),
and $(\mathfrak{h},\{h_i\}_{i\in I},\{\al_i\}_{i\in I})$
the root data (the realization of $A$),
that is,
\begin{itemize}
\item
$\mathfrak{h}$ is an $n$-dimensional vector space over $\C$,
and $\{h_i\}_i$ is a base of $\mathfrak{h}$;

\item
$\{\al_i\}_i$ is a base of $\mathfrak{h}^*$,
the space of linear functionals on $\mathfrak{h}$;

\item
$\al_j(h_i)=a_{ij}$ for all $i,j\in I$.
\end{itemize}

Each $\al_i$ is called a simple root.
The simple reflection
$s_i\col \mathfrak{h}^*\ra\mathfrak{h}^*$
is defined by
$s_i(\la):=\la-\la(h_i)\al_i$.
Note that $s_i^2=1$.
The Weyl group $W$
is the finite group generated by $s_i$'s.
The word length of $w\in W$ with respect to
$\{s_1,\dots,s_n\}$ is denoted by $\ell(w)$.
We denote by $w_0$ an element of maximal length.
It is known that
any $w\in W$ is contained in $w_0$,
that is,
$\ell(w_0)=\ell(w)+\ell(w^{-1}w_0)$.
It follows from this equality
that $w_0$ is unique and $w_0^2=1$.

Take positive integers $\{d_i\}_{i\in I}$
such that
$d_ia_{ij}=d_ja_{ji}$ for all $i,j\in I$.
A standard form is an inner product
$(\cdot,\cdot)$ on $\mathfrak{h}^*$
such that
$(\al_i,\al_j)=d_ia_{ij}$ for $i,j\in I$.
It is known that $(\al_i,\al_i)=2d_i$
can attain at most two values.
We normalize $\{d_i\}_{i\in I}$ so that
the smallest value of
$(\al_i,\al_i)$ is equal to 2.
Then
the normalized standard form
satisfies the following properties:
\begin{itemize}
\item
$W$-invariance,
that is,
$(w\la,w\mu)=(\la,\mu)$
for $w\in W$, $\la,\mu\in \mathfrak{h}^*$;

\item
$(\al_i,\al_i)/2\in \Z_+$ for $i\in I$;

\item
$\la(h_i)=2(\la,\al_i)/(\al_i,\al_i)$
for $i\in I$.
\end{itemize}

We associate $A$ with a finite dimensional
simple Lie algebra $\mathfrak{g}$ over $\C$
whose Cartan subalgebra is $\mathfrak{h}$.

\begin{defn}[Drinfel'd, Jimbo]
Let $0<q<1$.
The \emph{quantum universal enveloping algebra}
$U_q(\mathfrak{g})$
is the unital $\C$-algebra
generated by $\{K_i,X_i^+,X_i^-\}_{i\in I}$
such that
$K_i$ is invertible for all $i\in I$,
and the following relations hold
for all $i,j\in I$:
\begin{equation}
\label{eq:KX}
K_iK_j=K_jK_i,
\quad
K_iX_j^{\pm}=q^{\pm(\al_i,\al_j)/2}X_j^{\pm}K_i;
\end{equation}
\begin{equation}
X_i^+X_j^- -X_j^-X_i^+
=\de_{ij}
\frac{K_i^2-K_i^{-2}}{q_i-q_i^{-1}};
\end{equation}
\begin{equation}
\sum_{k=0}^{1-a_{ij}}
(-1)^k
\binom{1-a_{ij}}{k}_{q_i^2}q_i^{-k(1-a_{ij}-k)}
(X_i^{\pm})^kX_j^{\pm}(X_i^{\pm})^{1-a_{ij}-k}=0
\quad
\mbox{if }i\neq j,
\end{equation}
where $q_i:=q^{(\al_i,\al_i)/2}$,
and the $t$-binomial is defined by
\[
\binom{m}{n}_{t}
:=
\frac{(t;t)_m}{(t;t)_n(t;t)_{m-n}},
\quad
(a;t)_m=(1-a)(1-at)\cdots (1-at^{m-1})
\]
for $m,n\in\Z_+$ and $a\in\C$.
\end{defn}

Then $U_q(\mathfrak{g})$ has the Hopf$*$-algebra structure
defined as follows:
For $i\in I$,
\begin{itemize}
\item
(Coproduct)
\[
\De(K_i^{\pm})=K_i^{\pm}\oti K_i^{\pm},
\quad
\De(X_i^{\pm})=X_i^{\pm}\oti K_i+K_i^{-1}\oti X_i^{\pm},
\]

\item
(Antipode)
\[
S(K_i^{\pm})=K_i^{\mp},
\quad
S(X_i^{\pm})=-q_i^{\pm1}X_i^{\pm},
\]

\item
(Counit)
\[
\vep(K_i)=1,
\quad
\vep(X_i^{\pm})=0,
\]

\item
(Involution)
\begin{equation}
\label{eq:inv}
K_i^*=K_i,
\quad
(X_i^{\pm})^*=X_i^\mp.
\end{equation}
\end{itemize}

\subsection{Representation theory of $U_q(\mathfrak{g})$}
For a finite dimensional Hilbert space $H$,
we call a $*$-homomorphism $\pi$
from $U_q(\mathfrak{g})$ into $B(H)$
a \emph{representation}.
If the commutant of $\Ima \pi$ is trivial,
$\pi$ is said to be \emph{irreducible}.
In general,
$\Ima\pi$ is a finite dimensional C$^*$-algebra,
and $\pi$ is the direct sum of irreducibles.

Let $\pi\col U_q(\mathfrak{g})\to B(H)$
be an irreducible representation.
From (\ref{eq:inv}),
we obtain non-singular self-adjoint operators $\pi(K_i)$.
Let $p_i^{\pm}$ be the projections onto
the positive and negative spectral subspaces
of $\pi(K_i)$, respectively.
Then the relations (\ref{eq:KX}) and (\ref{eq:inv})
imply $p_i^\pm H$ are $U_q(\mathfrak{g})$-invariant.
Hence each $K_i$ has either only positive eigenvalues
or only negative ones.
We will say that
$\pi$ is \emph{admissible} when
$\pi(K_i)$ is a positive operator for all $i\in I$.

For $g=(g_k)_{k}\in \Z_2^n=\prod_{k=1}^n\{1,-1\}$,
we will define an automorphism $\si_g$ on $U_q(\mathfrak{g})$
as a $*$-algebra (not as a Hopf$*$-algebra)
by
\[
\si_g(K_i^\pm)=g_i K_i^\pm,
\quad
\si_g(X_i^\pm)=X_i^\pm
\quad
\mbox{for }
i\in I.
\]
Then for any irreducible $\pi$,
we can find $g\in \Z_2^n$
so that $\pi\circ\si_g$ is admissible.
Hence the classification of irreducible admissible representations
is essential.

Let $\pi\col U_q(\mathfrak{g})\to B(H)$
be a finite dimensional admissible representation.
Then $\pi(K_i)$'s are generating a commutative C$^*$-subalgebra.
Hence we obtain the decomposition of $H$ into the common eigenspaces.
For $\la\in\mathfrak{h}^*$,
we define
\[
H_\la:=\{\xi\in H\mid
\pi(K_i)\xi=q^{(\la,\al_i)/2}\xi
\}.
\]
Note that $q^{(\la,\al_i)/2}=q_i^{\la(h_i)/2}$.
Then the positivity of each $\pi(K_i)$
implies the direct sum decomposition of $H$:
\[
H=\bigoplus_{\la\in\mathfrak{h}^*}H_\la.
\]
If $H_\la\neq\{0\}$,
then $\la$ and $H_\la$ are called a \emph{weight}
and a \emph{weight space}
of $\pi$ (or of the $U_q(\mathfrak{g})$-module $H$),
respectively.
By $\Wt(H)$,
we denote the collection of the weights of $\pi$.
It is known that
$\dim H_\la=\dim H_{w\la}$
for all $\la\in \Wt(H)$ and $w\in W$.

If a non-zero vector $\xi\in H_\la$
is cyclic and $\pi(X_i^+)\xi=0$ for all $i\in I$,
then $\la$ and $\xi$ are called a \emph{highest weight}
and a \emph{highest weight vector} of $H$, respectively.
A lowest weight and a lowest vector are similarly introduced
by using $X_i^-$ instead of $X_i^+$.
It is known that
if $\la$ is a highest weight,
then $w_0\la$ is a lowest weight.

By $\Irr^+(U_q(\mathfrak{g}))$,
we denote the set of the unitary equivalence classes
of irreducible admissible representations
of finite dimension.

Let us introduce
the root lattice $Q:=\sum_{i\in I}\Z\al_i\subs\mathfrak{h}^*$,
and set $Q_+:=\sum_{i\in I}\Z_+\al_i$.
We will equip $\mathfrak{h}^*$
with the partial order $\leq$ such that
$\la\leq\mu$ if and only if $\mu-\la\in Q_+$.
Then it turns out that
a finite dimensional irreducible admissible representation has
a unique maximal weight
that is in fact the highest weight.
Moreover,
the weight space of the highest weight is one-dimensional.

Let
\[
P:=\{\la\in\mathfrak{h}^*\mid \la(h_i)\in\Z\mbox{ for all }i\in I\},
\]
and
\[
P_+:=\{\la\in\mathfrak{h}^*\mid \la(h_i)\in\Z_+\mbox{ for all }i\in I\}.
\]
We will call an element of $P$ and $P_+$
an \emph{integral weight} and
a \emph{dominant integral weight}, respectively.
When $\la\in P_+$ satisfies $\la(h_i)>0$ for all $i\in I$,
$\la$ is said to be \emph{regular}.
We will denote by $P_{++}$ the set of regular dominant weights.
Note that
$Q\subs P\subs \sum_{i\in I}\Q\al_i$
since $A$ is an invertible matrix.

From $\la\in P_+$,
we can construct an irreducible module $L(\la)$
called the \emph{Verma module}.
The highest weight theory says
that there exists a one-to-one correspondence
$P_+\ni \la \longleftrightarrow L(\la)
\in\Irr^+(U_q(\mathfrak{g}))$.
Let $\pi_\la\col U_q(\mathfrak{g})\to B(L(\la))$
be the corresponding representation.

Let $\om_i\in P_+$ be the \emph{fundamental weight}
that is defined by
$\om_i(h_j)=\de_{i,j}$ for $i,j\in I$.
Then $P=\sum_{i\in I}\Z\om_i$ and $P_+=\sum_{i\in I}\Z_+\om_i$.

The Weyl vector is defined by
$\rho:=(1/2)\sum_{\al\in \De_+}\al$,
where $\De:=\{w\al_i\mid w\in W,\ i\in I\}$
and $\De_+:=\De\cap Q_+$.
It is known that $\rho(h_i)=1$ for all $i\in I$,
that is, $\rho=\sum_{i\in I}\om_i$.
In particular, $\rho$ is a dominant integral weight.

\subsection{$C(G_q)$}
Let
$\pi\col U_q(\mathfrak{g})\ra B(H_\pi)$
be a finite dimensional representation.
For vectors $\xi,\eta\in H_\pi$,
we will define a linear functional
$C_{\xi,\eta}^\pi$ on $U_q(\mathfrak{g})$
by
\[
C_{\xi,\eta}^\pi(x):=\langle \pi(x)\eta,\xi\rangle
\quad
\mbox{for }
x\in U_q(\mathfrak{g}).
\]

For $\la\in P_+$ and $\mu\in \Wt(L(\la))$,
we will fix an orthonormal base
$\{\xi_\mu^i\mid i\in I_\mu^\la\}$
of $L(\la)_\mu$,
where $I_\mu^\la:=\{1,\dots,\dim L(\la)_\mu\}$.
We often write
$C_{\xi_\mu^i,\xi_\nu^j}^\la$
or, more simply,
$C_{\mu_i,\nu_j}^\la$
for $C_{\xi_\mu^i,\xi_\nu^j}^{\pi_\la}$.

Since $\dim L(\la)_{w\la}=\dim L(\la)_\la=1$
for $w\in W$,
we simply denote by $C_{w\la,w'\la}^\la$
for $C_{\xi_{w\la},\xi_{w'\la}}^\la$,
where $\xi_{w\la}\in L(\la)_{w\la}$
and $\xi_{w'\la}\in L(\la)_{w'\la}$
are fixed unit vectors for $w,w'\in W$.
Note that there exists an ambiguity
of a constant factor of modulus one
about this expression.

Then we will introduce
the following subspace of $U_q(\mathfrak{g})^*$:
\begin{align*}
A(G_q)
&:=
\spa
\{
C_{\xi,\eta}^\pi\mid
\xi,\eta\in H_\pi,
\
\pi\mbox{ is an admissible representation}
\}
\\
&\ =
\spa\{
C_{\mu_i,\nu_j}^\la
\mid
\mu,\nu\in \Wt(L(\la)),
\
i\in I_\mu^\la,
\
j\in I_\nu^\la,
\
\la\in P_+
\}.
\end{align*}

We put the Hopf$*$-algebra structure
on $A(G_q)$ as follows:
for $\ph,\ps\in A(G_q)$ and $x,y\in U_q(\mathfrak{g})$,
\begin{itemize}
\item
(Product)
\[
(\ph\ps)(x):=(\ph\oti\ps)(\De(x)),
\]
\item
(Coproduct)
\[
\de(\ph)(x\oti y):=\ph(xy),
\]
\item
(Antipode)
\[
\ka(\ph)(x):=\ph(S(x)),
\]
\item
(Counit)
\[
\vep(\ph):=\ph(1),
\]
\item
(Involution)
\[
\ph^*(x)=\ovl{\ph(S(x)^*)}.
\]
\end{itemize}
The following equality is frequently used:
\[
\de(C_{\mu_i,\nu_j}^\la)
=
\sum_{\zeta\in\Wt(\la),\,k\in I_\zeta^\la}
C_{\mu_i,\zeta_k}^\la\oti C_{\zeta_k,\nu_j}^\la,
\]
where $\Wt(\la)$ denotes $\Wt(L(\la))$.
The C$^*$-completion of $A(G_q)$ is denoted by $\CG$.
It is known that $G_q$ is co-amenable.
Hence the Haar state $h$ is faithful,
and the counit $\vep$ is norm bounded on $\CG$
(see, for example, \cite[Corollary 5.1]{Ban}).
We will describe the Woronowicz character of $A(G_q)$,
which is well-known for experts
(see \cite[Example 9, p.425]{Kl-Sch}).

\begin{lem}
\label{lem:Woronowicz}
For $z\in\C$,
the Woronowicz character
$f_z\col A(G_q)\ra\C$
is given by
\[
f_z(C_{\xi_\mu,\xi_\nu}^\la)
=
\langle\xi_\nu,\xi_\mu\rangle
q^{2(\mu,\rho)z}
\quad
\mbox{for all }
\xi_\mu\in L(\la)_\mu,
\
\xi_\nu\in L(\la)_\nu.
\]
\end{lem}
\begin{proof}
Let $\ph\in A(G_q)$.
Then
$\ta_{-i}(\ph)(x)=\ka^2(\ph)(x)=\ph(S^2(x))$
for $x\in \Uqg$.
If $x=X_{i_1}^-\cdots X_{i_k}^- K X_{j_1}^+\cdots X_{j_\ell}^+$
for some $i_1,\dots,i_k,j_1,\dots,j_\ell\in I$
and $K\in U_q(\mathfrak{h})$,
then
$S^2(x)=q_{i_1}^{-2}\cdots q_{i_k}^{-2}q_{j_1}^2\cdots q_{j_\ell}^2 x$,
where $U_q(\mathfrak{h})$ denotes the Hopf$*$-subalgebra
generated by $K_i$'s.
Since
$C_{\xi_\mu,\xi_\nu}^\la(x)$
can be non-zero
only if
$\mu=\nu+\al_{j_1}+\cdots+\al_{j_\ell}-\al_{i_1}-\cdots-\al_{i_k}$,
we obtain
$C_{\xi_\mu,\xi_\nu}^\la(S^2(x))=
q^{2(\rho,\mu-\nu)}
C_{\xi_\mu,\xi_\nu}^\la(x)$,
where we have used the equality $2(\rho,\al_i)=(\al_i,\al_i)$.
Thus
$
\ta_{-i}(C_{\xi_\mu,\xi_\nu}^\la)(x)
=q^{2(\rho,\mu-\nu)}
C_{\xi_\mu,\xi_\nu}^\la(x)$.
Since the map $\C\ni z\ra \ta_{iz}(a)\in A(G_q)$ is of exponential type
(see, for example, \cite[Proposition 11.11]{St}),
we obtain the following equality:
\[
\ta_t(C_{\xi_\mu,\xi_\nu}^\la)
=
q^{2(\mu-\nu,\rho)it}C_{\xi_\mu,\xi_\nu}^\la
\quad
\mbox{for all }t\in\R.
\]
We let $F_\la:=(\id\oti f_1)(C^\la)$.
Since
$(\id\oti \ta_t)(C^\la)=(F_\la^{it}\oti 1)C^\la (F_\la^{-it}\oti 1)$
and $C^\la$ is irreducible,
we have
$f_{it}(C_{\xi_\mu,\xi_\nu}^\la)
=\langle\xi_\nu,\xi_\mu\rangle q^{2(\mu,\rho)it}g^\la(it)$
for some entire function $g^\la$.
Thus $f_z\col A(G_q)\ra\C$ is given by
\[
f_z(C_{\xi_\mu,\xi_\nu}^\la)
=
\langle\xi_\nu,\xi_\mu\rangle
q^{2(\mu,\rho)z}g^\la(z)
\quad
\mbox{for all }
\xi_\mu\in L(\la)_\mu,
\
\xi_\nu\in L(\la)_\nu.
\]
Since $f_{z_1+z_2}=(f_{z_1}\oti f_{z_2})\circ\de$
for $z_1,z_2\in\C$,
we obtain $g^\la(z_1+z_2)=g^\la(z_1)g^\la(z_2)$.
Hence there exists $\th_\la\in\C$
such that $g^\la(z)=e^{\th_\la z}$
for all $z\in\C$.
The positivity of $F_\la^t=(\id\oti f_t)(C^\la)$
for $t\in\R$ implies that $\th_\la$ belongs to $\R$.
By definition of $f_z$,
we must have $\Tr(F_\la)=\Tr(F_\la^{-1})$.
Hence
\[
g^\la(1)
\sum_{\mu\in\Wt(\la)}q^{2(\mu,\rho)}\dim L(\la)_\mu
=
g^\la(-1)
\sum_{\mu\in\Wt(\la)}q^{-2(\mu,\rho)}\dim L(\la)_\mu.
\]
However, by $w_0^2=1$, $w_0\rho=-\rho$
and $\dim L(\la)_{\mu}=\dim L(\la)_{w_0\mu}$,
\begin{align*}
\sum_{\mu}q^{-2(\mu,\rho)}\dim L(\la)_\mu
&=
\sum_{\mu}q^{2(\mu,w_0\rho)}\dim L(\la)_\mu
=
\sum_{\mu}q^{2(w_0\mu,\rho)}\dim L(\la)_{\mu}
\\
&=
\sum_{\mu}q^{2(\mu,\rho)}\dim L(\la)_{w_0\mu}
=
\sum_{\mu}q^{2(\mu,\rho)}\dim L(\la)_\mu.
\end{align*}
Thus $g^\la(1)=g^\la(-1)$.
Hence $g^\la(1)^2=g^\la(1)g^\la(-1)=g^\la(0)=1$.
Since $g^\la(1)>0$,
$e^\th=g^\la(1)=1$.
This shows $\th_\la=0$ and $g^\la(z)=1$ for all $z\in\C$.
\end{proof}

Hence
for
$t\in\R$, $\la\in P_+$
and
$\mu,\nu\in\Wt(\la)$,
we have
\begin{equation}
\label{eq:modular}
\si_t^h(C_{\xi_\mu,\xi_\nu}^\la)
=
q^{2(\mu+\nu,\rho)it}
C_{\xi_\mu,\xi_\nu}^\la
\quad
\mbox{for all }
\xi_\mu\in L(\la)_\mu,
\
\xi_\nu\in L(\la)_\nu,
\end{equation}
and
the quantum dimension $\dim_q L(\la)$ is given by
\begin{equation}
\label{eq:quantum-dim}
\dim_q L(\la)
=
\sum_{\mu\in\Wt(\la)}q^{2(\mu,\rho)}\dim L(\la)_\mu.
\end{equation}

\subsection{Quantum subgroups of $G_q$}
Let us denote by $T$ the $n$-torus,
that is, $T=\{(t_1,\dots,t_n)\in\C^n\mid |t_i|=1,\ i=1,\dots,n\}$.
Recall that $n=|I|$.
By the coupling
$\langle\cdot,\cdot\rangle\col T\times P\ra\C$
defined by
\[
\langle(t_1,\dots,t_n),\mu\rangle
:=
t_1^{\mu(h_1)}\cdots t_n^{\mu(h_n)}
\quad
\mbox{for }
(t_1,\dots,t_n)\in T,
\
\mu\in P,
\]
we will regard $P$ as the dual group of $T$.
Let $\ev_t\col C(T)\ra\C$ be the evaluation map
at $t\in T$.
Then the restriction map $r_T\col \CG\ra\CT$ is defined
as follows:
\begin{equation}
\label{eq:rest}
\ev_t\circ\, r_T(C_{\xi_\mu,\xi_\nu}^\la)
=
\langle t,\mu\rangle
\langle\xi_\nu,\xi_\mu\rangle
\quad
\mbox{for all }t\in T,
\
\xi_\mu\in L(\la)_\mu,
\
\xi_\nu\in L(\la)_\nu.
\end{equation}
With this map, $T$ is a quantum subgroup of $G_q$
and called the \emph{maximal torus}.
Note that $r_T$ actually comes from the inclusion map
$U_q(\mathfrak{h})$ into $\Uqg$.

Let $\ga_t:=(\ev_t\circ\,r_T\oti\id)\circ\de$ be the left action
of $T$ on $\CG$.
Then we have
\begin{equation}
\label{eq:torusact}
\ga_t(C_{\xi_\mu,\xi_\nu}^\la)
=
\langle t,\mu\rangle
C_{\xi_\mu,\xi_\nu}^\la
\quad
\mbox{for all }t\in T,
\
\xi_\mu\in L(\la)_\mu,
\
\xi_\nu\in L(\la)_\nu.
\end{equation}

Let $U_{q_i}(su(2))$ be the Hopf$*$-subalgebra
of $\Uqg$ generated by
$\{1, K_i^{\pm1}, X_i^\pm\}$.
Then the canonical embedding
of $U_{q_i}(su(2))$ into $\Uqg$
gives a surjective $*$-homomorphism
$r_i\col \CG\to C(SU_{q_i}(2))$.
It turns out that $r_i$ is a restriction map.
So, $SU_{q_i}(2)$ is a quantum subgroup of $G_q$.

\subsection{Classification of irreducible representations of $C(G_q)$}
\label{subsect:class-irr-Soibelman}
In this subsection,
we recall theory of irreducible representations
of $C(G_q)$ as a C$^*$-algebra.

For $i\in I$,
set $H_{s_i}:=\ell^2(\Z_+)$
with $\{\vep_k\}_{k\in\Z_+}$, the standard orthonormal base.
Let us introduce an irreducible representation
$\pi_i\col C(SU_{q_i}(2))\ra B(H_{s_i})$ as follows:
\begin{equation}
\label{eq:pi-x-u}
\pi_i(x_i)\vep_k=\sqrt{1-q_i^{2k+2}}\vep_{k+1},
\quad
\pi_i(u_i)\vep_k=q_i^k\vep_k,
\end{equation}
\[
\pi_i(v_i)\vep_k=-q_i^{k+1}\vep_k,
\quad
\pi_i(y_i)\vep_k=\sqrt{1-q_i^{2k}}\vep_{k-1}
\quad
\mbox{for all }
k\geq0,
\]
where
\[
\begin{pmatrix}
x_i&u_i\\
v_i&y_i
\end{pmatrix}
:=
\begin{pmatrix}
r_i(C_{\om_i,\,\om_i}^{\om_i})
&
r_i(C_{\om_i,\,s_i\om_i}^{\om_i})
\\
r_i(C_{s_i\om_i,\,\om_i}^{\om_i})
&
r_i(C_{s_i\om_i,\,s_i\om_i}^{\om_i})
\end{pmatrix}
.
\]
Since $\pi_i(u_i)$ generates an atomic maximal abelian
subalgebra of $B(H_{s_i})$
and $\pi_i(x_i)$ is acting on $H_{s_i}$
as a weighted unilateral shift,
it turns out that $\pi_i$ is irreducible.

Note that $\pi(u_i)$ and $\pi(v_i)$ are compact operators.
Let $p_i$ be the quotient map from $B(H_{s_i})$
onto the Calkin algebra $B(H_{s_i})/K(H_{s_i})$.
Then we obtain $p_i(\pi_i(x_i))=p_i(S)$ and $p_i(\pi_i(u_i))=0$,
where $S$ denotes the unilateral shift on $H_{s_i}$.
This shows that $\Ima p_i\circ\pi_i$
coincides with the commutative C$^*$-algebra
${\rm C}^*(p_i(S))\cong C(S^1)$.

Let $\om_i\col \Ima p_i\circ\pi_i\ra\C$ be the character
defined by $\om_i(p_i(S))=1$.
Then the composition $\om_i\circ p_i\circ\pi_i$ is nothing but
the counit of $C(SU_{q_i}(2))$.
Let us denote by $\eta_i$ the map $\om_i\circ p_i$.
Then we obtain $\vep=\eta_i\circ\pi_i$.

Since $\pi_i$ is irreducible,
$\pi_{s_i}:=\pi_i\circ r_i$ is an irreducible representation of $\CG$.
The following theorem due to Soibel'man says that
those $\pi_{s_i}$ play a role of building blocks of
any irreducible representation of $\CG$.
See \cite[Theorem 3.4, 5.7]{S-repn}
or \cite[Theorem 5.3.3, 6.2.7, Chapter 3]{KS} for its proof.
Also see \cite[Theorem 4]{S-irred} and \cite[Theorem 3.1]{VS}
for the statement in the case of $SU_q(N)$.

For $w\in W$,
$I_w$ denotes the norm-closed ideal
in $\CG$ that is
generated by $C_{\xi,\xi_\la}^\la$
for $\la\in P_+$
such that $\langle U_q(\mathfrak{b}_+)L(\la)_{w\la},\xi\rangle=0$,
where $U_q(\mathfrak{b}_+)$ is the subalgebra of $U_q(\mathfrak{g})$
generated by $\{1,K_i^{\pm1},X_i^+\}_{i\in I}$.

\begin{thm}[Soibel'man]
\label{thm:irr-repn-G}
Let $G_q$ be as before.
\begin{enumerate}
\item
For any irreducible representation $\pi$ of $\CG$,
there exists a unique $w\in W$
such that $I_w\subs \ker\pi$;

\item
Let $\pi$ be an irreducible representation of $\CG$
corresponding to $w\in W$.
Then there exists $t\in T$
such that
$\pi$ is unitarily equivalent to the following representation:
\[
\pi_{t,w}:=(\pi_t\oti \pi_{s_{i_1}}\oti\cdots\oti \pi_{s_{i_k}})\circ\de^{(k)},
\]
where $\pi_t:=\ev_t\circ\, r_T$,
and $w=s_{i_1}\cdots s_{i_k}$ is a minimal expression of $w$.
The representation space is
$H_w:=H_{s_{i_1}}\oti\cdots \oti H_{s_{i_k}}$.
\end{enumerate}
\end{thm}

If $t=e$ (the neutral element of $T$),
then we will denote by $\pi_w$ for $\pi_{e,w}$.
By definition, we obtain $\pi_{t,w}=\pi_w\circ\ga_t$.

For a minimal expression $w=s_{i_1}\cdots s_{i_k}$,
we let $\eta_w:=\eta_{i_1}\oti\cdots \oti\eta_{i_k}$.
Then $\eta_w\circ\pi_w=\vep$.
By definition of $\de^{(n)}$ for $n\in\N$,
we have $\pi_w=(\pi_{w_1}\oti \pi_{w_2})\circ\de$
for $w=w_1w_2$ with $\ell(w)=\ell(w_1)+\ell(w_2)$.

\section{Quantum flag manifolds}
In this section,
we will study the quantum flag manifold $\CTG$ and
its measure theoretic analogue $\lTG$.
In particular,
we will prove the factoriality of $\lTG$.
Then we will present an alternative proof of
Dijkhuizen--Stokman's result for $T\subset G_q$.
Our proof involves knowledge of Poisson boundary,
but it is relatively short.
On recent development of this subject in a more general setting,
readers are referred to \cite{NeTu} and references therein.
As applications,
we will derive the quantum Weyl dimension formula of $L(\la)$
and also determine the intrinsic group of $G_{q,\Om}$,
the 2-cocycle deformation associated with $\lTG$.

\subsection{Classification of irreducible representations of $\CTG$}
\label{subsect:irreducible}

Let $\bG$ be a compact quantum group.
In \cite{T-Banach},
it is conjectured that
the classical Poisson boundary
$H_{\rm class}^\infty(\widehat{\bG},\mu)$
could coincide with
the center of the Poisson boundary
$Z(H^\infty(\widehat{\bG},\mu))$.
(See \cite{Iz-Poisson,INT,T-Poisson} for a detail
of theory of a Poisson boundary.)
In particular,
if the fusion rule of $\bG$ is commutative,
the conjecture asks
the factoriality of the Poisson boundary.
It has been verified
that the conjecture is true for $SU_q(2)$,
$A_o(F)$ \cite{TV,VV0}
and $A_u(F)$ \cite{VV}.
To these examples,
we can show that
a stronger property that
the canonical $\bG$-action on a Poisson boundary
is ``approximately inner.''
We will show the factoriality of $\lTG$ as follows.

For $\la\in P_+$,
we let $a_\la:=C_{\la,w_0\la}^\la$.
Then they generate a commutative C$^*$-subalgebra
of $\CG$.
See \cite[Corollary 2.1.5, Proposition 2.2.4, 2.3.2, Chapter 3]{KS}
for its proof.
Readers should note that
several typographical errors concerning signs are found
in \cite[Proposition 2.1.4, Corollary 2.1.5, Proposition 2.3.2, Chapter 3]{KS}.

\begin{thm}
\label{thm:factor}
The following statements hold:
\begin{enumerate}
\item 
$\ga$ is a faithful action of $T$ on the center
$Z(\lG)$;

\item
$L^\infty(T\backslash G_q)$
is the type I$_\infty$ factor.
\end{enumerate}
\end{thm}
\begin{proof}
(1).
Let $\la,\La\in P_+$, $\mu,\nu\in \Wt(\la)$,
$\xi_\mu\in L(\la)_\mu$ and $\xi_\nu\in L(\la)_\nu$.
Then by \cite[Corollary 2.1.5, Proposition 2.3.2, Chapter 3]{KS},
we have
\[
C_{\xi_\mu,\xi_\nu}^\la a_\La
=
q^{(\La,-\mu+w_0\nu)}a_\La C_{\xi_\mu,\xi_\nu}^\la,
\quad
C_{\xi_\mu,\xi_\nu}^\la a_\La^*
=
q^{(\La,-\mu+w_0\nu)}
a_\La^* C_{\xi_\mu,\xi_\nu}^\la.
\]
Thus
\begin{equation}
\label{eq:CaLa}
C_{\xi_\mu,\xi_\nu}^\la |a_\La|
=
q^{(\La,-\mu+w_0\nu)}
|a_\La| C_{\xi_\mu,\xi_\nu}^\la.
\end{equation}
Let $a_\La=v_\La|a_\La|$ be the polar decomposition
in $\lG$.
Then
$v_\La C_{\xi_\mu,\xi_\nu}^\la=C_{\xi_\mu,\xi_\nu}^\la v_\La$.
Hence $v_\La$ belongs to the center of $\lG$.
Then $\ga_t(v_\La)=\langle t, \La\rangle v_\La$
for $t\in T$.
Thus $\ga$ is faithful on $Z(\lG)$.

(2).
This is a direct consequence
of \cite[Theorem 4.7]{T-Banach}.
We will sketch the proof for readers' convenience.
Let $H^\infty(\widehat{G}_q)$ be the Poisson boundary
of $\widehat{G}_q$
and
$\Theta\col\lG\ra R(G_q)$
the Poisson integral
defined by
$\Theta(x):=(\id\oti h)(V^*(1\oti x)V)$
for $x\in\lG$.
Then $\Theta$
is a $\widehat{G}_q$-$G_q$-equivariant
isomorphism from $\lTG$ onto $H^\infty(\widehat{G}_q)$.
See \cite[Theorem 5.10]{Iz-Poisson},
\cite[Theorem A,B]{INT}
and
\cite[Corollary 4.11]{T-Poisson}
for the proof.

Let $x\in Z(\lG)\cap\lTG$.
Then $\Theta(x)$ is fixed by the coproduct of $R(G_q)$,
and $\Theta(x)$ is a scalar.
Hence $Z(\lG)^\ga=Z(\lG)\cap\lTG=\C$
which shows the central ergodicity of $\ga$.

By (1), $Z(\lG)$ is generated by unitaries $v_\la$,
$\la\in P_+$,
such that $\ga_t(v_\la)=\langle t,\la\rangle v_\la$
for $t\in T$.
From (\ref{eq:torusact}),
it turns out that $\lG=Z(\lG)\vee \lTG$.
Hence $Z(\lTG)\subs Z(\lG)\cap \lTG=\C$.
\end{proof}

Note that by ergodicity,
$v_\La$ in the above proof
must be a unitary.
As a corollary, we have the following.

\begin{cor}
The Poisson boundary $H^\infty(\widehat{G}_q)$
is the type I$_\infty$ factor.
\end{cor}

Now let us present our approach to
Dijkhuizen--Stokman's result.

\begin{lem}
\label{lem:faithful}
The representation $\pi_{w_0}$ is faithful on $C(T\backslash G_q)$.
\end{lem}
\begin{proof}
Let $x\in C(T\backslash G_q)$ with $\pi_{w_0}(x)=0$.
Let $w\in W\setminus\{w_0\}$.
Then $w':=w^{-1}w_0$ satisfies $\ell(w_0)=\ell(w)+\ell(w')$.
Hence we may assume that
$\pi_{w_0}=(\pi_{w}\oti\pi_{w'})\circ\de$.
Using the character $\eta_{w'}\col \Ima\pi_{w'}\ra\C$
with $\eta_{w'}\circ\pi_{w'}=\vep$,
we have
$(\id\oti\eta_{w'})\circ\pi_{w_0}=\pi_{w}$.
Thus $\pi_{w}(x)=0$ for any $w\in W$.
Since the family $\{\pi_{t,w}\}_{t\in T,w\in W}$ separates
the elements of $C(G_q)$
and $\pi_{t,w}=\pi_{w}$ on $C(T\backslash G_q)$,
we have $x=0$.
\end{proof}

\begin{lem}
For any $\la\in P_{++}$,
$|a_\la|^n$ uniformly converges
to the same minimal projection $p_0$ in $\CTG$
as $n\to\infty$.
\end{lem}
\begin{proof}
With a minimal expression $w_0=s_{i_1}\cdots s_{i_k}$,
we have the following (up to a factor of modulus one):
\begin{equation*}
\label{eq:alam}
\pi_{w_0}(a_\la)
=
\pi_{i_1}(u_{i_1})^{s_{i_2}\cdots s_{i_k}\la(h_{i_1})}
\oti
\cdots
\oti
\pi_{i_{k-1}}(u_{i_{k-1}})^{s_{i_k}\la(h_{i_{k-1}})}
\oti
\pi_{i_k}(u_{i_k})^{\la(h_{i_k})}.
\end{equation*}
(See \cite[p.120, (6.2.4)]{KS}.)
Since $\la$ is regular,
$s_{i_{\ell+1}}\cdots s_{i_k}\la(h_{i_\ell})$ is strictly positive
for each $\ell=1,\dots,k$.
Thus $\pi_{w_0}(|a_\la|)$
is a non-singular compact operator
with spectrum $\{0\}\cup \{q^m\}_{m\in S}\cup\{1\}$
for some $S\subs \N$.
Note that the eigenspace
with respect to the eigenvalue $1$
is the one-dimensional space $\C\vep_0^{\oti \ell(w_0)}$.
Hence we obtain
\[
\|\pi_{w_0}(|a_\la|^m)-\pi_{w_0}(|a_\la|^n)\|
\leq
q^m
\quad
\mbox{if }m\leq n.
\]
It follows from Lemma \ref{lem:faithful}
that $\||a_\la|^m-|a_\la|^n\|\leq q^m$.
In particular, 
$|a_\la|^n$ uniformly converges to an element $p_\la$
in $\CTG$.
Then $\pi_{w_0}(p_\la)$ is nothing but
the one rank projection onto $\C \vep_0^{\oti\ell(w_0)}$.
Since $\pi_{w_0}$ is faithful,
$p_\la$ does not depend on $\la$,
and we denote it by $p_0$.

Moreover, the following
equality shows the minimality of $p_0$ in $\CTG$.
\[
\pi_{w_0}(p_0 \CTG p_0)
=
\pi_{w_0}(p_0)\pi_{w_0}(\CTG)\pi_{w_0}(p_0)
=
\C \pi_{w_0}(p_0).
\]
\end{proof}

Now we consider a commutative C$^*$-algebra
$A:=\pi_{w_0}(C^*(|a_\la|\mid \la\in P_{++}))$
that is contained in $K(H_{w_0})$.
Since $A$ is non-degenerately acting on $H_{w_0}$,
there exists a sequence of projections
$q_\ell\in A$ for $\ell$ in some index set $J$
such that $q_0=\pi_{w_0}(p_0)$,
$\sum_{\ell\in J}q_\ell=1$ (strongly),
$\dim q_\ell H_{w_0}<\infty$
and
$\pi_{w_0}(|a_\la|)q_\ell
=\nu_{\la,\ell}q_\ell$ for some $\nu_{\la,\ell}>0$.
Using the faithfulness of $\pi_{w_0}$ on $\CTG$,
we take a family of orthogonal projections
$\{r_\ell\}_{\ell\in J}$
from $\CTG$ such that $\pi_{w_0}(r_\ell)=q_\ell$.
Also we must have
$|a_\la|r_\ell=\nu_{\la,\ell}r_\ell$.

Since $\pi_{w_0}$ is faithful,
$r_\ell \CTG r_\ell$
embeds into $q_\ell B(H_{w_0})q_\ell$,
which shows the finite dimensionality
of $r_\ell \CTG r_\ell$ for all $\ell\in J$.

\begin{lem}
The restriction of $\pi_{w_0}$ on $\CTG$
is irreducible.
\end{lem}
\begin{proof}
Since we know the minimality of $p_0$,
it suffices to show that
$\xi_0:=e_0^{\oti\ell(w_0)}$
is cyclic for $\pi_{w_0}(\CTG)$.
Take a vector $\xi\in H_{w_0}$
that is orthogonal to $\pi_{w_0}(\CTG)\xi_0$.
Thus we have
$\langle \xi,\pi_{w_0}(x)\xi_0\rangle=0$
for any $x\in \CTG$.
This implies that
$\pi_{w_0}(p_0 x)\xi=0$ for any $x\in \CTG$.

Now recall that $\lTG$ is a type I factor.
So, we have a faithful irreducible representation
$\ps\col \CTG\ra B(\ell^2)$.
Since each $r_\ell \CTG r_\ell$ is finite dimensional,
we know that $\Ima \ps$ has a non-zero intersection
with $K(\ell^2)$.
Hence $K(\ell^2)\subs \Ima \ps$.
This enables us to
take partial isometries $v_{i,\ell}$,
$i\in I_\ell$ in $\CTG$ such that
$v_{i,\ell}=r_\ell v_{i,\ell}p_0$
and
$r_\ell=\sum_{i\in I_\ell}v_{i,\ell}p_0v_{i,\ell}^*$
(finite sum).
Thus we have
$q_\ell\xi
=\sum_i \pi_{w_0}(v_{i,\ell}p_0v_{i,\ell}^*)\xi
=0$.
This shows $\xi=0$
because $\sum_{\ell\in J} q_\ell=1$.
\end{proof}

We will state Dijkhuizen--Stokman's result
for $T\subset G_q$ \cite[Theorem 5.9]{DS}.

\begin{thm}[Dijkhuizen--Stokman]
\label{thm:class-irr}
The following statements hold:
\begin{enumerate}
\item
For any $w\in W$,
the restriction of $\pi_w$ on $C(T\backslash G_q)$
gives an irreducible representation;

\item
Any irreducible representation
of $C(T\backslash G_q)$ is unitarily equivalent
to the restriction of $\pi_w$ with unique $w\in W$.
\end{enumerate}
\end{thm}
\begin{proof}
(1).
By the previous lemma,
it suffices to show the statement when $w\neq w_0$.
Let $w'=w^{-1}w_0$.
Then $\ell(w_0)=\ell(w)+\ell(w')$.
Let $L$ be a non-zero $\pi_w(\CTG)$-invariant
closed subspace.
Since $\pi_{w_0}(\CTG)
\subs \pi_w(\CTG)\oti \pi_{w'}(\CG)$,
$L\oti H_{w'}$
is invariant for $\pi_{w_0}(\CTG)$.
The irreducibility of $\pi_{w_0}(\CTG)$ shows that
$L\oti H_{w'}=H_{w_0}=H_w\oti H_{w'}$,
and $L=H_w$.

(2).
Let $\vph$ be a pure state on $\CTG$
and $\ps$ a pure state extension on $\CG$.
Then by Theorem \ref{thm:irr-repn-G} (2) due to Soibel'man,
there exists $t\in T$, $w\in W$ and
a unit vector $\xi\in H_{w}$
such that
$\ps(x)=\langle \pi_{t,w}(x)\xi,\xi\rangle$
for all $x\in\CG$.
By (1),
the restriction of $\pi_{t,w}$ on $\CTG$,
which coincides with $\pi_w$ on $\CTG$,
is irreducible.
Hence
$\xi$ is cyclic for $\pi_w(\CTG)$,
that is, $\vph$ comes from $\pi_w$.

Suppose that the restrictions of
$\pi_w$ and $\pi_{w'}$ on $\CTG$ are unitarily equivalent.
Then $\ker \pi_w \cap\CTG=\ker\pi_{w'}\cap\CTG$.
Thanks to Theorem \ref{thm:irr-repn-G} (1),
it suffices to show that $I_w\subs \ker \pi_{w'}$.
Take any $C_{\xi_\mu,\xi_\la}^\la$ in $I_w$
with $\xi_\mu\in L(\la)_\mu$ and $\mu\in\Wt(\la)$.
Then $\pi_w(|C_{\xi_\mu,\xi_\la}^\la|)=0$.
Since $|C_{\xi_\mu,\xi_\la}^\la|\in\CTG$,
we have $\pi_{w'}(|C_{\xi_\mu,\xi_\la}^\la|)=0$,
that is, $C_{\xi_\mu,\xi_\la}^\la\in \ker\pi_{w'}$.
Hence $I_w\subs \ker \pi_{w'}$,
and we are done.
\end{proof}

\begin{rem}
In the above proof,
we in fact have shown that
if $\pi_w$ and $\pi_{w'}$ are unitarily equivalent
as representations of $C(T\backslash G_q/T)$,
then $w=w'$.
\end{rem}

\begin{rem}
We can show $\pi_w$ with $w\neq w_0$
has the non-trivial kernel in $\CTG$
as follows.
Recall that for any $\la\in P_+$,
the lowest weight of $L(\la)$ is given by $w_0\la$.
In particular,
if $\mu\in \Wt(\la)$,
then $\mu\geq w_0\la$.
Take $\la\in P_+$
such that $w\la\neq w_0\la$.
Then $\pi_{w}(|C_{w_0\la,\la}^\la|)=0$,
which shows the non-triviality of
$I_w\cap \CTG$.
In particular,
the restriction of $\pi_{w_0}$ is conjugate
to the embedding $\iota\col \CTG\ra \lTG\cong B(\ell^2)$.
\end{rem}

The previous result immediately implies the following.

\begin{cor}
\label{cor:counit-unique}
The counit is the unique
character of $C(T\backslash G_q)$.
\end{cor}

For an action $\al$ of $G_q$
on a unital C$^*$-algebra (or von Neumann algebra)
$A$,
we denote by
$\Aut_{G_q}(A)$ the set
of $G_q$-equivariant automorphisms on $A$,
that is,
\[
\Aut_{G_q}(A)
:=
\{\th\in\Aut(A)\mid
\al\circ\th=(\th\oti\id)\circ\al
\}.
\]

\begin{thm}
The embedding of $\CTG$ into $\CG$ is unique.
\end{thm}
\begin{proof}
Suppose that $\ps\col \CTG\to\CG$ is a $G_q$-equivariant
embedding.
By Corollary \ref{cor:counit-unique},
we have $\vep\circ\ps=\vep$ on $\CTG$.
Then for $x\in\CTG$,
\[
x=(\vep\oti\id)(\de(x))
=(\vep\circ\ps\oti\id)(\de(x))
=(\vep\oti\id)(\de(\ps(x)))
=\ps(x).
\]
\end{proof}

From the previous result,
we have the following.

\begin{cor}\label{cor:equiv}
One has $\Aut_{G_q}(C(T\backslash G_q))=\{\id\}$.
\end{cor}

This also implies $\Aut_{G_q}(\lTG)=\{\id\}$
because an equivariant map on $\lTG$ preserves
each finite dimensional spectral subspace and also $\CTG$.

\subsection{Some formulae concerning the Haar state}
We will close this section computing the Haar state $h$ on $\CG$
in terms of a density operator.
The results obtained in this subsection
is not used in the subsequent sections.
Recall that $|a_\rho|=|C_{\rho,w_0\rho}^\rho|$
is a non-singular positive operator.

\begin{lem}
For $t\in\R$,
one has $\si_t^h=\Ad |a_\rho|^{2it}$ on $\lG$.
\end{lem}
\begin{proof}
Put $\La=\rho$ in (\ref{eq:CaLa}).
Then we have
$|a_\rho|^2 C_{\xi_\mu,\xi_\nu}^\la=q^{2(\mu+\nu,\rho)}
C_{\xi_\mu,\xi_\nu}^\la |a_\rho|^2$,
that is,
\[
(1\oti|a_\rho|^2)C^\la=(F_\la\oti1)C^\la(F_\la\oti|a_\rho|^2)
\quad
\mbox{for all }\la\in P_+.
\]
Since $C^\la\in B(L(\la))\oti \CG$ is a unitary,
we have
\[
(\id\oti |a_\rho|^{2it})C^\la
=
(F_\la^{it}\oti1)C^\la(F_\la^{it}\oti|a_\rho|^{2it})
\quad
\mbox{for all }\la\in P_+,\ t\in\R.
\]
Thus $\si_t^h=\Ad |a_\rho|^{2it}$.
\end{proof}

Let $\vph$ be the restriction of $h$ on $\lTG$.
Let $E_\ga\col\lG\to\lTG$ be the conditional expectation
defined by $E_\ga(x)=\int_T\ga_t(x)\,dt$,
where $dt$ denotes the normalized Haar measure on $T$.
Then $h\circ E_\ga=h$,
and
we have $\si_t^\vph=\si_t^h|_{\lTG}=\Ad|a_\rho|^{2it}$
for $t\in\R$.
Since $\lTG$ is a type I factor,
we obtain $\vph=c\Tr|a_\rho|^2$
for some $c>0$,
where $\Tr$ denotes the canonical trace of $\lTG$.

Recall the distinguished minimal projection $p_0$ in $\lTG$.
Using $|a_\rho|p_0=p_0$, we have $c=\vph(p_0)$.
Hence we get the following result which generalizes \cite[Lemma 2]{PV}.

\begin{thm}
\label{thm:Haar-state-density}
One has $h(x)=h(p_0)\Tr( |a_\rho|^2 E_\ga(x))$
for all $x\in\lG$.
In other words,
with the identification $\lG=L^\infty(T)\oti\lTG$,
one has
\[
h=
h(p_0)
\int_T\,dt\oti 
\Tr|a_\rho|^2.
\]
\end{thm}

We will compute $h(p_0)$.
Let $w_0=w_{i_1}\cdots w_{i_k}$ be a minimal expression.
We have known that $\pi_{w_0}\col \CTG\to B(H_{w_0})$
is conjugate to the embedding
$\iota\col \CTG\to \lTG$.
Let us identify $\lTG$ with $B(H_{w_0})$ in what follows.
Let $\la\in P_+$.
By (\ref{eq:alam}),
we have the following:
\[
h(|a_\la|^2)
=
h(p_0)
\cdot
\Tr\left(\pi_{i_1}(u_{i_1})^{2s_{i_2}\cdots s_{i_k}(\la+\rho)(h_{i_1})}\right)
\cdots
\Tr\left(\pi_{i_k}(u_{i_k})^{2(\la+\rho)(h_{i_k})}\right),
\]
where each $\Tr$ denotes the canonical trace of $B(\ell^2)$.
Using (\ref{eq:pi-x-u}) and $q_j=q^{(\al_j,\al_j)/2}$ for $j\in I$,
we obtain the following for $\ell=1,\dots,k$:
\[
\Tr\left(\pi_{i_\ell}
(u_{i_\ell})^{2s_{i_{\ell+1}}\cdots s_{i_k}(\la+\rho)
(h_{i_\ell})}\right)
=
\left(1-q^{2\left(\la+\rho,\,s_{i_k}\cdots s_{i_{\ell+1}}\al_{i_\ell}\right)}
\right)^{-1}.
\]
Since $\De_+=\{s_{i_k}\cdots s_{i_{\ell+1}}\al_{i_\ell}\mid \ell=1,\dots,k\}$,
we get
\begin{equation}
\label{eq:Haar-a-lambda}
h(|a_\la|^2)
=
h(p_0)
\prod_{\al\in\De_+}\left(1-q^{2(\la+\rho,\al)}\right)^{-1}.
\end{equation}
Putting $\la=0$ in the above,
we obtain $a_\la=1$ and the following result.

\begin{lem}
\label{lem:Haar-atom}
One has the following equality:
\[
h(p_0)=
\prod_{\al\in\De_+}\left(1-q^{2(\al,\rho)}\right).
\]
\end{lem}

For instance, if $G_q=SU_q(n)$,
we obtain
$h(p_0)=(q^2;q^2)_1\cdots(q^2;q^2)_{n-1}$.

Now we can derive the well-known
quantum Weyl dimension formula.
From (\ref{eq:Haar-ortho}) and Lemma \ref{lem:Woronowicz},
we obtain
\begin{align}
h(|a_\la|^2)
&=
h(C_{\la,w_0\la}^{\la}(C_{\la,w_0\la}^{\la})^*)
=
(\dim_q L(\la))^{-1} (F_{\la})_{w_0\la,w_0\la}
\notag\\
&=
(\dim_q L(\la))^{-1} q^{-2(\la,\rho)}.
\label{eq:Haar-a2n}
\end{align}
Employing (\ref{eq:Haar-a-lambda})
and Lemma \ref{lem:Haar-atom},
we get the following.

\begin{prop}[Quantum Weyl dimension formula]
\label{prop:quantum-dim}
For $\la\in P_+$, one has
\begin{align*}
\dim_q L(\la)&
=q^{-2(\la,\rho)}
\prod_{\al\in\De_+}\frac{\left(1-q^{2(\la+\rho,\al)}\right)}
{\left(1-q^{2(\rho,\al)}\right)}
\\
&=
\prod_{\al\in\De_+}
\frac{[(\la+\rho,\al)]_q}{[(\rho,\al)]_q},
\end{align*}
where $[n]_q$ denotes the $q$-integer
$(q^{-n}-q^n)/(q^{-1}-q)$ for $n\in\Z$.
\end{prop}

\begin{rem}
We can compute the value of $h(p_0)$ by using the Weyl character formula
and the fact
that $|a_\rho|^{2n}$ converges to $p_0$
as $n\to\infty$ in the norm topology.

Let $\ch(L(\La))$ be the character of $L(\La)$,
that is,
\begin{equation}
\label{eq:character}
\ch L(\La)=\sum_{\mu\in \Wt(\la)}\dim L(\La)_\mu e(\mu),
\end{equation}
where $e(\cdot)$ denotes the formal exponential
(see \cite[\S 10.2, p.172]{Kac}).
By the Weyl character formula
(see, for example, \cite[Theorem 10.4, p.173]{Kac}
or \cite[Proposition 12, p.203]{Kl-Sch}),
we obtain
\[
\ch L(\La)
=
\frac{\sum_{w\in W}
(-1)^{\ell(w)}e(w(\La+\rho)-\rho)}{\prod_{\al\in\De_+}(1-e(-\al))}.
\]

If we put $e(\mu)=q^{2(\mu,\rho)}$ in (\ref{eq:character}),
then it follows from (\ref{eq:quantum-dim})
that
$\ch L(\La)$ is equal to $\dim_q L(\La)$.
Thus
\[
\dim_q L(\La)
=
\frac{\sum_{w\in W}
(-1)^{\ell(w)}q^{2(w(\La+\rho)-\rho,\rho)}}
{\prod_{\al\in\De_+}\left(1-q^{-2(\al,\rho)}\right)}.
\]
Now we put $\La=n\rho$ in the equality above.
Then
\[
\dim_q L(n\rho)
=
q^{-2(\rho,\rho)}
\frac{\sum_{w\in W}
(-1)^{\ell(w)}q^{2(n+1)(w\rho,\rho)}}
{\prod_{\al\in\De_+}\left(1-q^{-2(\al,\rho)}\right)}.
\]

Recall that $w_0\rho$ is the lowest weight of $L(\rho)$.
In particular, $w_0\rho\leq w\rho$ for any $w\in W$
and the equality holds if and only if $w=w_0$.
Hence if $w\neq w_0$,
then $(w\rho,\rho)>(w_0\rho,\rho)=-(\rho,\rho)$,
and
$q^{2(n+1)(w\rho,\rho)}q^{2n(\rho,\rho)}\to0$
as $n\to\infty$.
From
$\vph(p_0)=\lim_n h(|a_\rho|^{2n})=\lim_n h(|a_{n\rho}|^2)$
and (\ref{eq:Haar-a2n}),
we have
\begin{align*}
h(p_0)
&=
q^{2(\rho,\rho)}
\lim_{n\to\infty}
\frac{\prod_{\al\in\De_+}\left(1-q^{-2(\al,\rho)}\right)}
{\sum_{w\in W}(-1)^{\ell(w)}q^{2(n+1)(w\rho,\rho)+2n(\rho,\rho)}}
\\
&=
q^{2(\rho,\rho)}
\lim_{n\to\infty}
\frac{\prod_{\al\in\De_+}\left(1-q^{-2(\al,\rho)}\right)}
{(-1)^{\ell(w_0)}q^{-2(\rho,\rho)}}
\\
&=
q^{4(\rho,\rho)}(-1)^{\ell(w_0)}\prod_{\al\in\De_+}
\left(1-q^{-2(\al,\rho)}\right).
\end{align*}
Using the fact that $\ell(w_0)$ equals $|\De_+|$,
the last term above is equal to
\begin{align*}
q^{4(\rho,\rho)}\prod_{\al\in\De_+}\left(q^{-2(\al,\rho)}-1\right)
&=
q^{4(\rho,\rho)}\prod_{\al\in\De_+}q^{-2(\al,\rho)}
\prod_{\al\in\De_+}\left(1-q^{2(\al,\rho)}\right)
\\
&=\prod_{\al\in\De_+}\left(1-q^{2(\al,\rho)}\right)
\end{align*}
since $2\rho=\sum_{\al\in\De_+}\al$.
Hence we are done.
\end{rem}

It is also possible to compute the value of the Haar state
at each diagonal minimal projection.
Let $m\in\Z_+$ and
$p_m\col \ell^2(\Z_+)\to \C\vep_m$ the minimal projection.
Then for $\ell=1,\dots,k$, we have
\[
\pi_{i_\ell}
(u_{i_\ell})^{2s_{i_{\ell+1}}\cdots s_{i_k}\rho
(h_{i_\ell})}
p_m
=
q^{2m\left(\rho,\,s_{i_k}\cdots s_{i_{\ell+1}}\al_{i_\ell}\right)}
p_m.
\]
Therefore, we obtain the following result.
\begin{prop}
For $m_1,\dots,m_k\in\Z_+$, one has
\[
h(p_{m_1}\oti\cdots\oti p_{m_k})
=
h(p_0)
\prod_{\ell=1}^k
q^{2m_\ell\left(\rho,\,s_{i_k}\cdots s_{i_{\ell+1}}\al_{i_\ell}\right)}.
\]
\end{prop}

\section{Invariant cocycles and 2-cocycle deformations}

In this section,
we will introduce the notion of invariant cocycles
evaluated in $\lTG$
and determine them.
As an application,
we will describe the intrinsic group
of a twisted quantum group $G_{q,\Om}$.

\subsection{Invariant cocycles evaluated in $\lTG$}
Recall that the left $T$-action $\ga$ is faithful and ergodic on $Z(\lG)$.
For each fundamental weight
$\om_i\in P$, $i\in I$,
we take a unitary $v_{\om_i}$ in $Z(\lG)$
such that
$\ga_t(v_{\om_i})=\langle t,\om_i\rangle v_{\om_i}$
for $t\in T$.
Next,
for $\la=\sum_{i=1}^n a_i\om_i\in P$, $a_i\in\Z$,
we set
$v_\la:=v_{\om_1}^{a_1}\cdots v_{\om_n}^{a_n}$.
Then we have
\begin{equation}
\label{eq:vla}
\ga_t(v_\la)=\langle t,\la\rangle v_\la,
\quad
v_\la v_\mu=v_{\la+\mu}
\quad
\mbox{for }
t\in T,\ \la,\mu\in P.
\end{equation}
Note that $v_\la$'s are generating $Z(\lG)$
as a von Neumann algebra,
and $\lG=Z(\lG)\vee\lTG$.

Now for each $\la\in P$,
we introduce the unitary $w_\la$
defined by:
\begin{equation}
\label{eq:cocycle}
\de(v_\la)=(v_\la\oti1)w_\la.
\end{equation}
Since $\de$ and $\ga$ are commuting,
$w_\la$ belongs to $\lTG\oti \lG$.
From the above equation,
the cocycle relation $(w_\la\oti1)(\de\oti\id)(w_\la)=(\id\oti\de)(w_\la)$
holds.
Moreover, setting $\de^{w_\la}:=\Ad w_\la\circ\de$,
we have $\de^{w_\la}=\de$ on $\lG$
because $v_\la$ is a central unitary.

Let us denote by
$Z_{\rm inv}^1(\de,\lTG)$
the collection of $\lTG$-valued cocycles $w$
such that $\de^w=\de$ on $\lTG$.
In general,
when $\al$ is an action of $G_q$
on a von Neumann algebra $\cN$,
we denote by $Z_{\rm inv}^1(\al,\cN)$
the set of $\al$-cocycles $w$
with $\al^w=\al$.
We call such $w$ an \emph{invariant cocycle}.

We know $w_\la$'s are invariant cocycles of
the action of $\de$ on $\lTG$,
but in fact, they are all.

\begin{thm}
\label{thm:descript-coc}
In the above setting,
the following statements hold:
\begin{enumerate}
\item
$Z_{\rm inv}^1(\de,\lTG)
=\{w_\la\mid \la\in P\}$;

\item
The map $P\ni \la\mapsto w_\la \in Z_{\rm inv}^1(\de,\lTG)$
is a group isomorphism.
\end{enumerate}
\end{thm}
\begin{proof}
(1).
Let $w\in Z_{\rm inv}^1(\de,\lTG)$.
First, we will observe that
the twisted right action $\de^w$ is ergodic on $\lG$.
Let $x\in \lG^{\de^w}$.
Since $w\in \lTG\oti\lG$,
the actions $\ga$ and $\de^w$ are commuting.
Thus we may and do assume that
$\ga_t(x)=\langle t,\mu\rangle x$
for some $\mu\in P$ and all $t\in T$.
Then $\ga$ fixes $y:=v_\mu^*x$,
and $y$ belongs to $\lTG$.
Using $w\de(x)w^*=\de^w(x)=x\oti1$
and $\de^w=\de$ on $\lTG$,
we have
\begin{equation}
\label{eq:wvmuy}
ww_\mu w^*\de(y)
=
w(v_\mu^*\oti1)\de(v_\mu)\cdot\de(y)w^*
=
(v_\mu^*\oti1)w\de(x)w^*
=y\oti1.
\end{equation}
In particular,
$y^*y$ and $yy^*$ are fixed by $\de$
since $ww_\mu w^*$ commutes with $\de(yy^*)$.
By ergodicity of $\de$ on $\lTG$,
we may and do assume that $y$ is a unitary.

Set the inner automorphism $\ps:=\Ad y^*$ on $\lTG$.
Then for $z\in\lTG$,
\begin{align*}
\de(\ps(z))
=
\de(y^*)\de(z)\de(y)
&=
(y^*\oti1)ww_\mu w^*\cdot
\de(z)\cdot ww_\mu^*w^*(y\oti1)
\quad
\mbox{by }(\ref{eq:wvmuy})\\
&=
(y^*\oti1)\de(z)(y\oti1)
=
(\ps\oti\id)(\de(z)).
\end{align*}
Namely, $\ps\in \Aut_{G_q}(\lTG)$.
However, we must have $\ps=\id$
by Corollary \ref{cor:equiv}.
It turns out from Theorem \ref{thm:factor}
that $y\in\C$.
From (\ref{eq:wvmuy}), we have $w_\mu=1$.
This shows $\mu=0$.
(See the proof of the second statement.)
Hence $x$ is a scalar,
and $\de^w$ is ergodic on $\lG$.

Second,
we will use Connes' $2\times2$ matrix trick.
Let $\cN:=M_2(\C)\oti \lG$
and $\{e_{ij}\}_{i,j=1}^2$ be a system of matrix units
of $M_2(\C)$.
We set $\al:=\id\oti\de$ and $\ovl{w}:=e_{11}\oti1+e_{22}\oti w$.
Then $\ovl{w}$ is an $\al$-cocycle.
We will show that
the projections
$p_1:=e_{11}\oti1$ and $p_2:=e_{22}\oti1$
are Murray--von Neumann equivalent
in $\cN^{\al}$.

Consider the crossed product $\cN\rti_{\al} G_q$.
Since $\al$ is a cocycle perturbation
of $\id\oti\de$,
$\cN\rti_{\al} G_q$ is canonically isomorphic to $M_2(\C)\oti \lG\rti_\de G_q$,
and also to $M_2(\C)\oti B(L^2(G_q))$.
Thus $\cN\rti_{\al} G_q$ is the type I$_\infty$ factor.
The fixed point algebra $\cN^{\al}$ is a corner of $\cN\rti_{\al}G_q$,
and $\cN^{\al}$ is a type I factor.
Since $p_1\cN^{\al} p_1=\lG^\de=\C$
and $p_2 \cN^{\al} p_2=\lG^{\de^w}=\C$,
$p_1$ and $p_2$ are minimal projections in $\cN^{\al}$.
Hence they are equivalent.

Let us take a unitary $v\in\lG$
such that
$e_{12}\oti v\in \cN^\al$,
that is, $\de(v)=(v\oti1)w$.
The ergodicity of $\de$ shows that $v$ is the unique solution
of this equation up to a scalar multiple.
Indeed, if $v'$ is a (not necessarily unitary) another solution,
then $\de(v'v^*)=(v'\oti1)w\cdot w^*(v^*\oti1)=v'v^*\oti1$,
and $v'v^*\in\C$.
Then by the commutativity
of $\ga$ and $\de$ and also the equality
$(\ga_t\oti\id)(w)=w$,
we can find a unique $\la\in P$ such that
$v=v_\la v_1$ for some unitary $v_1\in\lTG$.
However,
the invariance $\de^w=\de$ on $\lTG$
deduces $\Ad v_1\in \Aut_{G_q}(\lTG)$.
Hence $v_1$ is a scalar as before,
and $w=w_\la$.

(2).
Let $\la,\mu\in P$.
Since $v_\la$ is central and $v_\la v_\mu=v_{\la+\mu}$,
we have
\[
w_\la w_\mu
=
(v_\la^*\oti1)\de(v_\la)\cdot (v_\mu^*\oti1)\de(v_\mu)
=
(v_\la^*\oti1)(v_\mu^*\oti1)\de(v_\la)\cdot \de(v_\mu)
=w_{\la+\mu}.
\]

To show the injectivity, let $w_\la=1$.
Then $v_\la$ is fixed by $\de$,
and $v_\la$ is a scalar.
Hence $\langle t,\la\rangle=1$
for all $t\in T$,
and $\la=0$.
\end{proof}

\subsection{2-cocycle deformations}
As we have shown,
the quantum flag manifold $\lTG$ is a type I factor.
So, we would like to find a unitary
which implements the right action
$\de$ on $\lTG$.

\begin{lem}
\label{lem:U}
There exists a unitary $U\in\lTG\oti\lG$
such that
$\de(x)=U(x\oti1)U^*$
for all $x\in\lTG$.
\end{lem}
\begin{proof}
Let $p_0$ be the distinguished minimal projection
of $\CTG$ as before.
Actually,
$p_0$ is also contained in $C(G_q/T)$,
and
$\de(p_0)\in\CTG\oti C(G_q/T)$.
Recall that $\lTG\oti L^\infty(G_q/T)$
is a type I factor.
We will show that
$\de(p_0)$ is an infinite projection.

Suppose that
$\de(p_0)$ is of finite rank.
Then $\de(p_0)(\CTG\oti C(G_q/T))\de(p_0)$
is finite dimensional.
Since the restrictions of $\pi_{w_0}$
on $\CTG$ and $C(G_q/T)$ are irreducible,
$e:=(\pi_{w_0}\oti\pi_{w_0})(\de(p_0))$
is a finite rank projection
of $B(H_{w_0})\oti B(H_{w_0})$.
In particular,
$e$ is a compact operator
on $H_{w_0}\oti H_{w_0}$.
Then we obtain
$(\id_{K(H_{w_0})}\oti \eta_{w_0})(e)=0$
since $K(H_{w_0})
=K(H_{s_{i_1}})\oti\cdots \oti K(H_{s_{i_k}})$
for a minimal expression $w_0=s_{i_1}\cdots s_{i_k}$.
However, since $\eta_{w_0}\circ\pi_{w_0}=\vep$,
we obtain $\pi_{w_0}(p_0)=0$.
It follows that $p_0=0$, a contradiction.
Hence $\de(p_0)$ is an infinite projection
of $\lTG\oti L^\infty(G_q/T)$.

Take a partial isometry $v\in \lTG\oti L^\infty(G_q/T)$
such that
$v^*v=p_0\oti1$ and $vv^*=\de(p_0)$.
Let $\{e_{ij}\}_{i,j=0}^\infty$
be a system of matrix units
of $\lTG$ such that $e_{00}=p_0$.
Then a unitary
$U:=\sum_{i=0}^\infty \de(e_{i0})v(e_{0i}\oti1)$
implements $\de$.
\end{proof}

\begin{rem}
By the same proof as the above,
we can show that
$(\pi_{w_0}\oti\pi_{w})(\de(p_0))$
is a projection of infinite rank
on $H_{w_0}\oti H_{w}$ for any $w\in W\neq\{e\}$.
\end{rem}

The coassociativity of $\de$,
implies that $(\id\oti\de)(U)^*U_{12}U_{13}$
commutes with $x\oti1\oti1$ for all $x\in\lTG$.
By factoriality of $\lTG$,
we obtain a unitary $\Om\in\lG\oti\lG$
such that
\[
U_{12}U_{13}=(\id\oti\de)(U)(1\oti\Om^*).
\]
Then $\Om$ satisfies the following 2-cocycle relation:
\[
(\Om\oti1)(\de\oti\id)(\Om)
=
(1\oti\Om)(\id\oti\de)(\Om).
\]
Denote by $\de_\Om$ the twisted coproduct $\Ad\Om\circ\de$.
Then thanks to \cite[Theorem 6.2]{DC},
the pair $(\lG,\de_\Om)$
becomes a new, in general non-compact,
locally compact quantum group in the sense of \cite{KV},
which we will denote by $G_{q,\Om}$.
Set $L^\infty(G_{q,\Om}):=L^\infty(G_q)$.
Readers are referred to \cite{DC-proj} for a general treatment
of projective representations.

We will not study a concrete description of $G_{q,\Om}$,
but describe a group-like elements.
A unitary $u\in L^\infty(G_{q,\Om})$
is called a \emph{group-like} element
when
$\de_\Om(u)=u\oti u$.
Denote by $\sG(G_{q,\Om})$ the collection of all group-like elements
of $G_{q,\Om}$,
which is called the \emph{intrinsic group} of $G_{q,\Om}$.

For a unitary $u\in L^\infty(G_{q,\Om})$,
we will put
\begin{equation}
\label{eq:wu}
w^u:=U(1\oti u)U^*\in \lTG\oti\lG.
\end{equation}

\begin{lem}
The unitary $w^u$ is a $\de$-cocycle
if and only if $u\in\sG(G_{q,\Om})$.
\end{lem}
\begin{proof}
We have
\begin{align*}
(w^u\oti1)(\de\oti\id)(w^u)
&=
U_{12} u_2 U_{12}^*
\cdot
U_{12}(w^u)_{13}U_{12}^*
\\
&=
U_{12} u_2 U_{13}u_3U_{13}^*U_{12}^*
\\
&=
(\id\oti\de)(U)\Om_{23}^*
(1\oti u\oti u)
\Om_{23}(\id\oti\de)(U^*),
\end{align*}
and
\begin{align*}
(\id\oti\de)(w^u)
&=
(\id\oti\de)(U)
(1\oti\de(u))
(\id\oti\de)(U^*).
\end{align*}
Hence we are done.
\end{proof}

By Lemma \ref{lem:U} and (\ref{eq:wu}),
$w^u\de(x)=\de(x)w^u$ for all $x\in \lTG$.
Thus $w^u$ belongs to $Z_{\rm inv}^1(\de,\lTG)$
if $u\in\sG(G_{q,\Om})$.

\begin{thm}
\label{thm:intrinsic}
The map $\sG(G_{q,\Om})\ni u\mapsto w^u\in Z_{\rm inv}^1(\de,\lTG)$
is a group isomorphism.
In particular,
$\sG(G_{q,\Om})$ is isomorphic to $P=\widehat{T}$.
\end{thm}
\begin{proof}
It is trivial that this map is a injective group homomorphism.
To show the surjectivity,
let $w\in Z_{\rm inv}^1(\de,\lTG)$.
Then for $x\in\lTG$,
we have $w\de(x)w^*=\de(x)$,
and so $U^*wU$ commutes with $x\oti1$.
By factoriality of $\lTG$,
there exists a unitary $u\in \lG$
such that $w=U(1\oti u)U^*$.
Since $w$ is a $\de$-cocycle,
$u$ is group-like with respect to $\de_\Om$
by the previous lemma.
The remaining statement follows from Theorem \ref{thm:descript-coc}.
\end{proof}

This result shows that
we have a Hopf algebra embedding of
$L^\infty(T)$ into $L^\infty(G_{q,\Om})$,
that is,
the $n$-dimensional torus $T$
is a ``quotient quantum group'' of $G_{q,\Om}$.

When $G_q=SU_q(2)$,
it has been proved that
$G_{q,\Om}$ is isomorphic to $\widetilde{E}_q(2)$,
Woronowicz's quantum $E(2)$ group
in \cite[Theorem 4.5]{DC-E2}.
In \cite[Theorem 2.1]{W-E2},
Woronowicz has classified
unitary representations of $\widetilde{E}_q(2)$.
In particular,
the intrinsic group of $\widetilde{E}_q(2)$
is generated by
the canonical unitary representation $v$ (see \cite[p.254, (1)]{W-E2}),
and is indeed isomorphic to $\Z\cong\widehat{T}$.

\section{Product type actions}
In this section,
we will study a product type action of $G_q$.
We will fix our notations.

Let $v$ be a unitary representation
of $G_q$ on a Hilbert space $H_v$
with $2\leq\dim H_v\leq\infty$.
Take a faithful invariant state $\ph\in B(H_v)_*$,
that is,
$\ph$ satisfies
$(\ph\oti\id)(v(x\oti1)v^*)=\ph(x)1$
for all $x\in B(H_v)$.
Note that $\ph$ is never tracial
since $(\id\oti f_1)(C^\la)\neq1$
for any non-zero $\la\in P_+$.

Consider the infinite tensor product
$(\cM,\vph):=\otimes_{m=1}^\infty(B(H_v),\ph)''$
that is a factor of type III.
The Connes' $S$-invariant is computed from the period
of $\ph$.
Let $\al\col \cM\ra \cM\oti \lG$ be
the product type action with respect to $v$
\cite{Iz-Poisson,INT,KNW}.
Let $E_\al\col\cM\ra\cM^\al$ be the conditional expectation
defined by $E_\al:=(\id\oti h)\circ\al$.
Then $\vph\circ E_\al=\vph$.

\subsection{Depth 2 inclusions}
In what follows, we always assume that
$\al$ is faithful,
that is,
the subspace $\al(\cM)(\cM\oti\C)$
is dense in $\cM\oti\lG$.
This is the case when each irreducible representation
of $G_q$
is contained
in a product unitary representations
$(v\oti \ovl{v})^{\oti m}$ for some $m\in\N$.
Then as remarked in \cite[p.509]{INT},
each irreducible is contained in
$v^{\oti m}$ for some $m\in\N$.
Therefore,
the faithfulness of $\al$
implies the generating property of the corresponding
probability measure on $\Irr(G_q)$.

Thanks to \cite[Corollary 3.9]{Iz-Poisson},
the relative commutant $\cQ:=(\cM^\al)'\cap\cM$
is non-trivial.
Moreover,
we know by \cite[Theorem 5.10]{Iz-Poisson},
\cite[Theorem A]{INT} and \cite[Corollary 4.11]{T-Poisson}
that there exists a $G_q$-equivariant isomorphism
from $\lTG$ onto $\cQ$.
In particular, $\cQ$ is a type I factor
by Theorem \ref{thm:factor}.
Then we have the following
tensor product decomposition:
\begin{equation}
\label{eq:tensor}
\cM=\cR\vee\cQ\cong \cR\oti\cQ,
\end{equation}
where $\cR:=\cQ'\cap\cM$.
Note that the invariant state $\vph$ is of the form $\vph|_\cR\oti\vph|_\cQ$.
Indeed, the modular automorphism group
$\si^\vph$ preserves $\cM^\al$,
and it does $\cQ$.
Hence,
by Takesaki's theorem \cite[p.309]{Tak},
there exists a unique conditional expectation
$F\col\cM\ra\cQ$ with $\vph\circ F=\vph$.
Then $F$ maps $\cR$ into the center $Z(\cQ)=\C$,
and
\[
\vph(xy)=\vph(F(x)y)=\vph(F(x))\vph(y)=\vph(x)\vph(y)
\quad
\mbox{for }
x\in\cR,\ y\in\cQ.
\]
We will study the inclusion $\cM^\al\subs\cR$.

\begin{lem}
\label{lem:depth2}
The inclusion $\cM^\al\subs\cR$ is irreducible and of depth 2.
\end{lem}
\begin{proof}
First we will show the irreducibility.
Let $x\in (\cM^\al)'\cap\cR$.
Then $x\in (\cM^\al)'\cap\cM=\cQ$,
but $x\in\cR=\cQ'\cap\cM$.
Hence $x\in Z(\cQ)=\C$.

Next we let $\cM^\al\subs\cR\subs\cR_1\subs\cR_2$ be the Jones tower.
We will show that $(\cM^\al)'\cap \cR_2$ is a type I factor.
Let $\cM^\al\subs\cM\subs\cM_1\subs\cM_2$ be the Jones tower.
By (\ref{eq:tensor}),
this is isomorphic to the following:
\[
\cM^\al\oti\C\subs \cR\oti \cQ\subs \cR_1\oti \cQ_1
\subs
\cR_2\oti \cQ_2,
\]
where $\cQ_1$ and $\cQ_2$ are type I factors.
Thus it suffices to show that
$(\cM^\al)'\cap \cM_2$ is a type I factor.

\begin{clm}
The Jones tower $\cM^\al\subs\cM\subs \cM_1\subs \cM_2$
is isomorphic to
$\cM^\al\oti\C\subs \al(\cM)\subs \cM\rti_\al G_q\subs \cM\oti B(L^2(G_q))$.
\end{clm}
\begin{proof}[Proof of Claim]
Since $\al$ is integrable,
there exists a canonical surjection from $\cM\rti_\al G_q$ onto
$\cM_1$ (see, for example, \cite[p.510]{INT} or \cite[Theorem 5.3]{V}).
Recall that $\al$ is faithful.
Thanks to \cite[Corollary 1.5]{INT},
we have an isomorphism from $\al(\cM)'\cap (\cM\rti_\al G_q)$
onto $\cM'\cap\cM_1$.
In particular, the canonical surjection from
$\cM\rti_\al G_q$ onto $\cM_1$ is an isomorphism.
Hence the tower $\cM^\al\subs\cM\subs\cM_1$
is isomorphic to $\cM^\al\oti\C\subs \al(\cM)\subs \cM\rti_\al G_q$.
The basic extension of $\al(\cM)\subs \cM\rti_\al G_q$
is realized as $\cM\oti B(L^2(G_q))$ through
the computation of
the modular conjugation of $\cM\rti_\al G_q$.
See \cite[Proof of Proposition 5.9]{V}
or \cite[Lemma 5.7]{Iz-can} for its proof.
\end{proof}

Hence $(\cM^\al)'\cap\cM_2$ is isomorphic to
$((\cM^\al)'\cap\cM)\oti B(L^2(G_q))
=\cQ\oti B(L^2(G_q))$,
which is a type I factor.
\end{proof}

The restriction of $E_\al$ on $\cR$ gives a conditional expectation
from $\cR$ onto $\cM^\al$.
Therefore,
there exists a compact quantum group and its minimal action $\be$
on $\cR$
with the fixed point algebra $\cM^\al$.
(See \cite{EN,L,Sz}.
It is worth mentioning that
Longo's sector-approach still works in our case.)
We will show that the quantum group
is nothing but the maximal torus $T$.

Note that
$\cR^\be$ can be of type II$_1$
though  $\cR$ is of type III.
Thus $\be$ is not dual in general
(see \cite[Proposition 5.2 (5)]{Iz-can}).
However, the action is semidual
(see, for example, \cite[Theorem 4.4]{NT},
\cite[Proposition 6.4]{V}
and
\cite[Theorem 2.2]{Y}).
So let us take the tensor product by $B(\ell^2)$
as follows:
\[
\ovl{\cM}:=B(\ell^2)\oti\cM,
\quad
\bal:=\id\oti\al.
\]
Also set $\ovl{\cR}:=B(\ell^2)\oti\cR$
and $\ovl{\cQ}:=\C\oti\cQ$.
Then we have
\[
\ovl{\cM}=\ovl{\cR}\vee\ovl{\cQ}\cong\ovl{\cR}\oti\ovl{\cQ},
\quad
\ovl{\cM}^{\bal}=B(\ell^2)\oti\cM^\al,
\quad
(\ovl{\cM}^{\bal})'\cap \ovl{\cM}
=
\ovl{\cQ}.
\]
Let $\pi$ be a $G_q$-equivariant isomorphism
from $\lTG$ onto $\ovl{\cQ}$,
which is unique by Corollary \ref{cor:equiv}.

Recall $w_\la\in Z_{\rm inv}^1(\de,\lTG)$
introduced in (\ref{eq:cocycle}).
Set an $\bal$-cocycle $w_\la^o$ defined by
\begin{equation}
\label{eq:wlao}
w_\la^o:=(\pi\oti\id)(w_\la).
\end{equation}
Then $w_\la^o\in Z_{\rm inv}^1(\bal,\ovl{\cQ})$
for all $\la\in P$.

Note that
$\ovl{\cM}\rti_{\bal}G_q$ is an infinite factor.
Since $\ovl{\cM}^{\bal}$ and $\ovl{\cM}^{\bal^{w_\la^o}}$
contain $B(\ell^2)\oti\C$,
they are also infinite factors.
Using the $2\times2$ matrix trick as before,
we obtain a unitary $u_\la\in\ovl{\cM}$
for $\la\in P$
such that
$\bal(u_\la)=(u_\la\oti1)w_\la^o$.
For $x\in\ovl{\cM}^\bal$,
we have
\[
\bal(u_\la x u_\la^*)
=
(u_\la\oti1)w_\la^o(x\oti1)(w_\la^o)^*(u_\la^*\oti1)
=
u_\la xu_\la^*\oti1.
\]
Hence $\th_\la:=\Ad u_\la$
gives an endomorphism on $\ovl{\cM}^\bal$.
We will show that $\th_\la$ is in fact an automorphism.

\begin{lem}
For any $\la\in P$,
$u_\la$ belongs to $\ovl{\cR}$.
\end{lem}
\begin{proof}
Let $x\in\ovl{\cQ}$ and $y\in\ovl{\cM}^\bal$.
Then
\[
u_\la^* x u_\la y
=
u_\la^* x\th_\la(y)u_\la
=
u_\la^* \th_\la(y)xu_\la
=
yu_\la^* x u_\la.
\]
Hence $\rho_\la:=\Ad u_\la^*$
defines an endomorphism on $\ovl{\cQ}$.
Moreover for $x\in \ovl{\cQ}$,
we have
\begin{align*}
\bal(\rho_\la(x))
&=
\bal(u_\la^*)\bal(x)\bal(u_\la)
=
(w_\la^o)^*(u_\la^*\oti1)\bal(x)(u_\la\oti1)w_\la^o
\\
&=
(w_\la^o)^*(\rho_\la\oti\id)(\bal(x))w_\la^o.
\end{align*}
Thus
\begin{align*}
\bal(\rho_\la(x))
&=
w_\la^o\bal(\rho_\la(x))(w_\la^o)^*
\quad
\mbox{because }
w_\la^o\in Z_{\rm inv}^1(\bal,\ovl{\cQ})
\\
&=
w_\la^o\cdot
(w_\la^o)^*(\rho_\la\oti\id)(\bal(x))w_\la^o
\cdot
(w_\la^o)^*
\\
&=
(\rho_\la\oti\id)(\bal(x)).
\end{align*}
Namely,
$\rho_\la$ is a $G_q$-equivariant embedding of
$\ovl{\cQ}$ into itself.
The injectivity of $\rho_\la$
implies that
spectral multiplicities of
$\rho_\la(\ovl{\cQ})$ and $\ovl{\cQ}$ must coincide,
and $\rho_\la(\ovl{\cQ})=\ovl{\cQ}$.
Hence $\rho_\la$ is a $G_q$-equivariant
automorphism on $\ovl{\cQ}\cong\lTG$.
By Corollary \ref{cor:equiv},
we obtain $\rho_\la=\id$,
that is,
$u_\la\in \ovl{\cQ}'\cap\ovl{\cM}=\ovl{\cR}$.
\end{proof}

Let $\la,\mu\in P$.
Since $u_\mu$ belongs to $\ovl{\cR}$,
we see $w_\la^o(u_\mu\oti1)=(u_\mu\oti1)w_\la^o$,
and
\begin{align*}
\bal(u_\la u_\mu)
&=
(u_\la\oti1)w_\la^o(u_\mu\oti1)w_\mu^o
=
(u_\la u_\mu\oti1)w_\la^o w_\mu^o
\\
&=
(u_\la u_\mu\oti1)w_{\la+\mu}^o
=
(u_\la u_\mu u_{\la+\mu}^*\oti1)
\bal(u_{\la+\mu}).
\end{align*}
It follows that
$c_{\la,\mu}:=u_\la u_\mu u_{\la+\mu}^*$
is contained in $\ovl{\cM}^\bal$,
and
\begin{equation}
\label{eq:cocycle-action}
\th_\la\circ\th_\mu
=
\Ad c_{\la,\mu}\circ\th_{\la+\mu}
\quad
\mbox{for all }
\la,\mu\in P.
\end{equation}

We will show that
$(\th,c)$ is a cocycle action of $P$
on $\ovl{\cM}^\bal$ as below.
(See \cite{Ocn}
for the definition of a cocycle action.)
If we put $\mu=-\la$,
we have $u_{\la+\mu}=u_0=1$,
and $\th_\la\circ\th_{-\la}=\Ad c_{\la,-\la}$.
This shows the surjectivity of $\th_\la$,
that is,
$\th_\la\in\Aut(\ovl{\cM}^\bal)$.

The 2-cocycle identity of $c$
is verified as follows:
for $\la,\mu,\nu\in P$,
\begin{align*}
c_{\la,\mu}c_{\la+\mu,\nu}
&=
u_\la u_\mu u_{\la+\mu}^*
\cdot
u_{\la+\mu} u_\nu u_{\la+\mu+\nu}^*
\\
&=
u_\la u_\mu u_\nu u_{\la+\mu+\nu}^*,
\end{align*}
and
\begin{align*}
\th_\la(c_{\mu,\nu})u_{\la,\mu+\nu}
&=
u_\la\cdot
u_\mu u_\nu u_{\mu+\nu}^*
\cdot
u_\la^*
\cdot
u_\la u_{\mu+\nu}u_{\la+\mu+\nu}^*
\\
&=
u_\la u_\mu u_\nu u_{\la+\mu+\nu}^*.
\end{align*}
From (\ref{eq:cocycle-action}),
it turns out that $(\th,c)$ gives a cocycle action
on an infinite factor $\ovl{\cM}^\bal$.
Then $c$ is in fact a coboundary
by \cite[Proposition 2.1.3]{Su-homoII}.
Take unitaries $u_\la'$ in $\ovl{\cM}^\bal$
for $\la\in P$
so that
$u_\la'\th_\la(u_\mu')c_{\la,\mu}(u_{\la+\mu}')^*=1$
for $\la,\mu\in P$.
By replacing $u_\la$ with $u_\la'u_\la$
if necessary,
we may and do assume that
our $u_\la$'s are satisfying
\begin{equation}
\label{eq:ula}
u_\la\in\ovl{\cR},
\quad
u_\la u_\mu=u_{\la+\mu},
\quad
\bal(u_\la)=(u_\la\oti1)w_\la^o
\quad
\mbox{for all }
\la,\mu\in P.
\end{equation}

Then we have an outer action
$\th$ of $P$ on $\ovl{\cM}^\bal$.
Indeed, if for some $\la\in P$,
$a\in\ovl{\cM}^\bal$ satisfies
$ax=\th_\la(x)a$ for all $x\in\ovl{\cM}^\bal$,
then $u_\la^*a\in (\ovl{\cM}^\bal)'\cap\ovl{\cR}=\C$.
This, however, implies that $u_\la\in \ovl{\cM}^\bal$
and $w_\la^o=1$, that is, $\la=0$.
We will prove that $\ovl{\cR}$ is actually generated by
$\ovl{\cM}^\bal$ and $\{u_\la\}_{\la\in P}$
by sector technique developed in \cite{ILP}.

Since the inclusion $\cN:=\ovl{\cM}^\bal\subs \ovl{\cR}$
comes from a compact quantum group action,
$\ovl{\cR}$ is the crossed product of
$\cN$
by the dual discrete quantum group action.
We would like to determine this action.
For the sake of this,
we should study the sector
$[\ga_\cN^{\ovl{\cR}}|_{\cN}]$
in $\Sect(\cN)$,
where $\ga_\cN^{\ovl{\cR}}$ denotes
the canonical endomorphism
from $\ovl{\cR}$ into $\cN$.
Since there exists a conditional expectation
from $\ovl{\cM}$ onto $\ovl{\cR}$,
we have a canonical embedding
${}_\cN L^2(\ovl{\cR})_\cN\subs
{}_\cN L^2(\ovl{\cM})_\cN$.
Hence $[\ga_\cN^{\ovl{\cR}}|_{\cN}]$
is contained in
$[\ga_\cN^{\ovl{\cM}}|_\cN]$.
So let us study ${}_\cN L^2(\ovl{\cM})_\cN$ first.
The author thanks Izumi for suggesting this method.

\begin{lem}
\label{lem:Hilb-embed}
For any $\la\in P_+$,
there exists a Hilbert space $\sH_\la$
with $s(\sH_\la)=1$ in $\ovl{\cM}$
which admits
an isometric $G_q$-isomorphism
from $L(\la)$ onto $\sH_\la$.
\end{lem}
\begin{proof}
Recall that $\cM^\al$ is not a type I factor.
Otherwise,
it follows that
$\cM=\cM^\al\vee\cQ$,
and $\cM$ would be of type I.
Thus we can take
a von Neumann algebra embedding
$\ps$
of $B(L(\la))$ into
$\C1_{\ell^2}\oti \cM^\al\subs \ovl{\cM}^\bal$.
Let $w:=(\ps\oti\id)(C^\la)$.
Then $w$ is an $\bal$-cocycle,
and $B(\ell^2)\oti\C$ is contained in
the fixed point algebra of $\ovl{\cM}$
by $\bal^w$.
Hence we can employ the $2\times2$-matrix trick
as usual,
we obtain a unitary $u\in \ovl{\cM}$
such that
$\bal(u)=(u\oti1)w$.

Let $\{e_{ij}\}_{i,j\in I}$
be a system of matrix units of
$\Ima \ps$
such that each $e_{ii}$ is minimal.
Fix an element $i_0$ in $I$.
Since $1_{\ell^2}\oti e_{i_0 i_0}$
is infinite projection,
there exists an isometry
in $V_{i_0}\in B(\ell^2)\oti\cM^\al$
such that
$V_{i_0}V_{i_0}^*=1_{\ell^2}\oti e_{i_0 i_0}$.
For $i\in I$,
we set $V_i:=(1\oti e_{ii_0})V_{i_0}$.
Then it turns out that
$V_i^*V_j=\de_{ij}1$
and
$V_iV_j^*=1\oti e_{ij}$
for $i,j\in I$.
We let $W_i:=uV_i$.
Then
\begin{align*}
\bal(W_i)
&=
\bal(u)\bal(V_i)
=
(u\oti1)w(V_i\oti1)
\\
&=
\sum_{j\in I}
(W_j\oti1)(V_j^*\oti1)w(V_i\oti1).
\end{align*}
Therefore,
the statement follows by setting
$\sH_\la:=\spa \{W_i\mid i\in I\}$.
\end{proof}

For $\la\in P_+$,
let $T_\la$ be an isometric $G_q$-isomorphism
from $L(\la)$ onto a Hilbert space $\sH_\la$
in $\ovl{\cM}$.
We let $V_{\mu_i}^\la:=T_\la(\xi_{\mu_i})$
for $\mu\in \Wt(\la)$
and $i\in I_\mu^\la$.
Then we obtain
\[
\bal(V_{\mu_i}^\la)
=\sum_{\nu\in\Wt(\la),\,j\in I_\nu^\la}
V_{\nu_j}^\la\oti C_{\nu_j,\mu_i}^\la.
\]
For $\la,\La\in P_+$,
$\mu\in\Wt(\la)$, $i\in I_\mu^\la$,
$\nu\in\Wt(\La)$ and $j\in I_\nu^\La$,
we have
\begin{align*}
E_\bal(V_{\mu_i}^\la (V_{\nu_j}^\La)^*)
&=
(\id\oti h)
(\bal(V_{\mu_i}^\la (V_{\nu_j}^\La)^*))
\\
&=
\sum_{\eta,\zeta,k,\ell}
V_{\eta_k}^\la (V_{\zeta_\ell}^\La)^*
\cdot
h(C_{\eta_k,\mu_i}^\la(C_{\zeta_\ell,\nu_j}^\la)^*)
\\
&=
\de_{\la,\La}
\de_{\mu,\nu}\de_{i,j}
(\dim_q L(\la))^{-1}F_{\mu_i,\mu_i}^\la
\sum_{\eta,\zeta,k,\ell}
\de_{\eta,\zeta}\de_{k,\ell}
V_{\eta_k}^\la (V_{\zeta_\ell}^\La)^*
\quad
\mbox{by }
(\ref{eq:Haar-ortho})
\\
&=
\de_{\la,\La}
\de_{\mu,\nu}\de_{i,j}
(\dim_q L(\la))^{-1}F_{\mu_i,\mu_i}^\la
\sum_{\eta,k}
V_{\eta_k}^\la (V_{\eta_k}^\La)^*
\\
&=
\de_{\la,\La}
\de_{\mu,\nu}\de_{i,j}
(\dim_q L(\la))^{-1}F_{\mu_i,\mu_i}^\la
\\
&=
\de_{\la,\La}
\de_{\mu,\nu}\de_{i,j}
(\dim_q L(\la))^{-1}q^{2(\mu,\rho)}
\quad
\mbox{by Lemma }
\ref{lem:Woronowicz}.
\end{align*}
So if we put
$W_{\mu_i}^\la:=(\dim_q L(\la))^{1/2}q^{-(\mu,\rho)}V_{\mu_i}^\la$,
then $\{W_{\mu_i}^\la\}_{\mu,i}$ is a linear base
of $\sH_\la$ such that
\begin{equation}
\label{eq:orthogonal}
E_\bal(W_{\mu_i}^\la(W_{\nu_j}^\La)^*)
=\de_{\la,\La}\de_{\mu,\nu}\de_{i,j}.
\end{equation}

Now we let
\[
\si_\la(x)
:=
\sum_{\mu\in\Wt(\la),\,i\in I_\mu^\la}V_{\mu_i}^\la x (V_{\mu_i}^\la)^*
\quad
\mbox{for }
x\in\cN.
\]
Then $\si_\la$ is an endomorphism on $\cN$.
We will determine the intertwiner space
$(\si_\la,\si_\la)$.
Recall $v_\mu\in Z(\lG)$
and
$w_\mu\in Z_{\rm inv}^1(\de,\lTG)$
introduced in (\ref{eq:vla}) and (\ref{eq:cocycle}).
Then we set
\[
a_{\mu_i}^\la
:=
\sum_{\nu\in \Wt(\la),\,j\in I_\nu^\la}
V_{\nu_j}^\la \pi((C_{\mu_i,\nu_j}^\la)^*v_\mu)
u_\mu^*
\quad
\mbox{for }
\la\in P_+,
\
\mu\in\Wt(\la),
\
i\in I_\mu^\la,
\]
where we recall that
$u_\mu$ is satisfying (\ref{eq:ula}).
Note that $(C_{\mu_i,\nu_j}^\la)^*v_\mu$
is contained in $\lTG$,
and $\pi((C_{\mu_i,\nu_j}^\la)^*v_\mu)$
is well-defined.

\begin{lem}
Let $\la\in P_+$.
Then the following statements hold:
\begin{enumerate}
\item
For all $\mu\in\Wt(\la)$,
$\{a_{\mu_i}^\la\}_{i\in I_\mu^\la}$
is an orthonormal base of $(\th_\mu,\si_\la)$;

\item
$(\si_\la,\si_\la)$
is linearly spanned by
$a_{\mu_i}^\la(a_{\mu_j}^\la)^*$
for $\mu\in\Wt(\la)$ and $i,j\in I_\mu^\la$.
\end{enumerate}
\end{lem}
\begin{proof}
(1).
We will check that
$a_{\mu_i}^\la$ is contained in $\ovl{\cM}^\bal$.
Indeed,
\begin{align*}
\bal(a_{\mu_i}^\la)
&=
\sum_{\nu,j}
\bal(V_{\nu_j}^\la)
\cdot
(\pi\oti\id)(\de((C_{\mu_i,\nu_j}^\la)^*v_\mu))
\cdot
\bal(u_\mu^*)
\\
&=
\sum_{\nu,\eta,\zeta,j,k,\ell}
(V_{\eta_k}^\la\oti C_{\eta_k,\nu_j}^\la)
\cdot
(\pi\oti\id)
\left(
((C_{\mu_i,\zeta_\ell}^\la)^*
\oti
(C_{\zeta_\ell,\nu_j}^\la)^*)
(v_\mu\oti1)w_\mu
\right)
\\
&\quad\quad
\cdot
(w_\mu^o)^*(u_\mu^*\oti1)
\\
&=
\sum_{\nu,\eta,\zeta,j,k,\ell}
\de_{\eta,\zeta}\de_{k,\ell}
(V_{\eta_k}^\la\oti 1)
\cdot
(\pi\oti\id)
\left(
((C_{\mu_i,\zeta_\ell}^\la)^*
\oti
1)
(v_\mu\oti1)
\right)
\cdot
(u_\mu^*\oti1)
\\
&=
a_{\mu_i}^\la\oti1.
\end{align*}

Next we have the following for $x\in \ovl{\cM}^\bal$:
\begin{align*}
a_{\mu_i}^\la \th_\mu(x)
&=
\sum_{\nu,j}
V_{\nu_j}^\la \pi((C_{\mu_i,\nu_j}^\la)^*v_\mu)
u_\mu^*
\th_\mu(x)
\\
&=
\sum_{\nu,j}
V_{\nu_j}^\la \pi((C_{\mu_i,\nu_j}^\la)^*v_\mu)
x
u_\mu^*
\\
&=
\sum_{\nu,j}
V_{\nu_j}^\la x \pi((C_{\mu_i,\nu_j}^\la)^*v_\mu)
u_\mu^*
=
\si_\la(x)a_{\mu_i}^\la.
\end{align*}
Thus $a_{\mu_i}^\la\in(\th_\mu,\si_\la)$.
We have
\begin{align*}
(a_{\mu_i}^\la)^*a_{\nu_j}^\la
&=
\sum_{\eta,\zeta,k,\ell}
u_\mu
\pi(v_\mu^*C_{\mu_i,\eta_k}^\la)
(V_{\eta_k}^\la)^*
\cdot
V_{\zeta_\ell}^\la \pi((C_{\nu_j,\zeta_\ell}^\la)^*v_\nu)
u_\nu^*
\\
&=
\sum_{\eta,k}
u_\mu
\pi(v_\mu^*C_{\mu_i,\eta_k}^\la
(C_{\nu_j,\eta_k}^\la)^*v_\nu)
u_\nu^*
\\
&=
\de_{\mu,\nu}\de_{i,j},
\end{align*}
and
\begin{align*}
\sum_{\mu,i}
a_{\mu_i}^\la(a_{\mu_i}^\la)^*
&=
\sum_{\mu,\eta,\zeta,i,k,\ell}
V_{\eta_k}^\la \pi((C_{\mu_i,\eta_k}^\la)^*v_\mu)
u_\mu^*
\cdot
u_\mu
\pi(v_\mu^*C_{\mu_i,\zeta_\ell}^\la)
(V_{\zeta_\ell}^\la)^* 
\\
&=
\sum_{\eta,\zeta,k,\ell}
\de_{\eta,\zeta}\de_{k,\ell}
V_{\eta_k}^\la
(V_{\zeta_\ell}^\la)^* 
=1.
\end{align*}
Hence for $x\in\ovl{\cM}^\bal$,
we obtain
\[
\si_\la(x)
=
\sum_{\mu,i}
a_{\mu_i}^\la\th_\mu(x)(a_{\mu_i}^\la)^*,
\]
and we are done.
\end{proof}

By the previous lemma,
we get the following equality in $\Sect(\cN)$:
\begin{equation}
\label{eq:decomp-sigma}
[\si_\la]
=\bigoplus_{\mu\in\Wt(\la)} \dim L(\la)_\mu [\th_\mu].
\end{equation}

\begin{lem}
In $\Sect(\cN)$,
one has
\[
[\ga_\cN^{\ovl{\cM}}|_\cN]
=
\bigoplus_{\la\in P_+}
\dim L(\la)[\si_\la].
\]
\end{lem}
\begin{proof}
We will decompose the $\cN$-$\cN$ bimodule
${}_\cN L^2(\ovl{\cM})_\cN$ as follows.
First we observe that
the linear span of
$\sH_\la^* \cN$, $\la\in P_+$
is weakly dense in $\ovl{\cM}$.
Indeed,
for any $\la\in P$
and any equivariant map
$T\col L(\la)\to \ovl{\cM}$,
it turns out that
$a:=\sum_{\mu,i}V_{\mu_i}^\la T(\xi_{\mu_i}^\la)^*$
belongs to $\ovl{\cM}^\bal$.
It follows that
$T(\xi_{\mu_i}^\la)^*=(V_{\mu_i}^\la)^*a\in\sH_\la^*\cN$.
Since the linear span of
$T(\xi_{\mu_i}^\la)^*$'s for $T$ and $\la,\mu,i$
is weakly dense in $\ovl{\cM}$,
we are done.

Next recall the $\al$-invariant state $\vph$
on $\cM$.
Take a faithful normal state $\ps$ on $B(\ell^2)$
and put $\ovph:=\ps\oti\vph$.
Then $\ovph$ is an $\bal$-invariant state
on $\ovl{\cM}$.
For $\la\in P_+$,
we let $e_N\col L^2(\cM)\ra \ovl{\cN 1_\ovph}$
be the Jones projection.
Then from (\ref{eq:orthogonal}),
$z_\la:=\sum_{\mu,i}(W_{\mu_i}^\la)^*e_N W_{\mu_i}^\la$
is a projection onto
the subspace $X_\la:=\sH_\la^*\ovl{\cN1_\ovph}$.
Since each
$(W_{\mu_i}^\la)^*e_N W_{\mu_i}^\la$
belongs to $\cN'\cap J_\ovph\cN' J_\ovph$,
the subspace $(W_{\mu_i}^\la)^* \ovl{\cN1_\ovph}$
is an $\cN$-$\cN$-bimodule.
Thus we have the following decomposition
as $\cN$-$\cN$-bimodules:
\begin{equation}
\label{eq:decomp-bimodule}
{}_\cN L^2(\cM)_\cN
=
\bigoplus_{\la\in P_+}
\bigoplus_{\mu\in \Wt(\la),\, i\in I_\mu^\la}
(W_{\mu_i}^\la)^* \ovl{\cN1_\ovph}.
\end{equation}

Let us consider the map
$\cN1_\ovph\ni x1_\ovph\mapsto (W_{\mu_i}^\la)^* x1_\ovph$.
Again by (\ref{eq:orthogonal}),
it turns out that this map extends to
the unitary map $U$ from $\ovl{\cN1_\ovph}$
onto $(W_{\mu_i}^\la)^* \ovl{\cN1_\ovph}$.
Then for $a,b\in\cN$ and $\xi\in\ovl{\cN1_\ovph}$,
we have
\begin{align*}
U(\si_\la(a)\xi b)
&=
(W_{\mu_i}^\la)^*\cdot(\si_\la(a)\xi b)
\\
&=
(W_{\mu_i}^\la)^*\si_\la(a)\xi b
\\
&=
a((W_{\mu_i}^\la)^*\xi) b
=
a(U\xi)b.
\end{align*}
Hence as $\cN$-$\cN$-bimodules,
$(W_{\mu_i}^\la)^* \ovl{\cN1_\ovph}$
and
${}_{\cN\,\si_\la}L^2(\cN)_\cN$
are isomorphic.
Then the statement follows
from (\ref{eq:decomp-bimodule}).
\end{proof}

By (\ref{eq:decomp-sigma}) and the previous lemma,
we obtain
\[
[\ga_\cN^{\ovl{\cM}}|_\cN]
=
\bigoplus_{\mu\in P}
\infty[\th_\mu].
\]
Since $[\ga_\cN^{\ovl{\cR}}|_\cN]$
is contained in
$[\ga_\cN^{\ovl{\cM}}|_\cN]$,
$[\ga_\cN^{\ovl{\cR}}|_\cN]$ is a direct sum of
multiples of $[\th_\mu]$'s.
Now we know $\cN\subs \ovl{\cR}$ comes from a minimal action
of a compact quantum group.
Since every $\th_\mu$ is an automorphism,
each irreducible representation of the compact quantum group
is one-dimensional.
Thanks to \cite[p.49]{ILP} or \cite[Lemma 3.5]{T-Gal},
we have
\[
[\ga_\cN^{\ovl{\cR}}|_\cN]
=
\bigoplus_{\mu\in S}
[\th_\mu]
\quad
\mbox{for some }
S\subs P.
\]
However, each $u_\mu$ is actually an element of $\ovl{\cR}$
which implements $\th_\mu$,
$[\th_\mu]$ must be contained in $[\ga_\cN^{\ovl{\cR}}|_\cN]$.
Thus we obtain $S=P$, that is,
\[
[\ga_\cN^{\ovl{\cR}}|_\cN]
=
\bigoplus_{\mu\in P}
[\th_\mu].
\]
Then by \cite[Theorem 3.9]{ILP},
it turns out that $\ovl{\cR}$ is generated by
$\cN=\ovl{\cM}^\bal$ and $u_\mu$, $\mu\in P$.
For $\mu\in P$,
$E_\bal(u_\mu)$ is an element
in $(\id,\th_\mu)$
The outerness of $\th$
implies that $E_\bal(u_\mu)=0$ for $\mu\neq0$.
Hence we obtain the following result.

\begin{thm}
\label{thm:depth2-torus}
The inclusion $\ovl{\cM}^\bal\subs\ovl{\cR}$
is isomorphic to
$\ovl{\cM}^\bal\subs \ovl{\cM}^\bal\rti_\th \widehat{T}$,
where $\widehat{T}=P$ as usual.
\end{thm}

\begin{rem}
Recall the unitary $U$ introduced
in Lemma \ref{lem:U}.
Let $\Ga(x):=(\pi\oti\id)(U^*)\al(x)(\pi\oti\id)(U)$
for $x\in\cM$.
Then $(\Ga\oti\id)\circ\Ga=(\id\oti\de_\Om)\circ\Ga$,
that is,
$\Ga$ is an action of $G_{q,\Om}$ on $\cR$.
However, $\Ga$ is not faithful.
Indeed,
$\cR^\Ga=\cM^\al$
and $\Ga(u_\la)=u_\la\oti \nu_\la$ for $\la\in P$,
where $\nu_\la$ is a group-like unitary
of $G_{q,\Om}$
such that
$U^*w_\la U=1\oti \nu_\la$.
Hence $\Ga$ is nothing but the dual action
of $\th$.
\end{rem}

\subsection{Induced actions}
Let us introduce the action
$\be$ on $\ovl{\cR}$,
that is, 
\[
\be_t(u_\mu)
=\langle t,\mu\rangle u_\mu
\quad
\mbox{for all }
t\in T,\ \mu\in P.
\]
By definition, $\be_t=\hat{\th}_{t^{-1}}$,
where $\hat{\th}$ denotes
the dual action of $\th$.
Then $\be$ extends to $\ovl{\cM}$
by putting $\be=\id$ on $\ovl{\cQ}$.
Then $\ovph\circ\be_t=\ovph$ for all $t\in T$
since $E_\bal(u_\mu)=0$ if $\mu\neq0$.
We will show that
$W^*(u_\mu\mid\mu\in P)\vee \cQ$
is naturally isomorphic to $\lG$.

\begin{lem}
There exists a von Neumann algebra isomorphism
$\pi\col\lG\ra
W^*(u_\mu \mid\mu\in P)\vee \cQ$
such that
\begin{itemize}
\item
$\bal\circ\pi=(\pi\oti\id)\circ\de$;

\item
$\be_{t}\circ\pi=\pi\circ\ga_t$
for all $t\in T$;

\item
$\pi(v_\mu)=u_{\mu}$
for $\mu\in P$;

\item
$\pi(\lTG)=\cQ$.
\end{itemize}
\end{lem}
\begin{proof}
Let $\pi\col \lTG\ra\cQ$ be a $G_q$-equivariant
isomorphism as before.
Let $w_\mu,w_\mu^o$
be the invariant cocycles defined in
(\ref{eq:cocycle}) and (\ref{eq:wlao}).
They are satisfying the following equalities:
\[
\de(v_\mu)=(v_\mu\oti1)w_\mu,
\quad
\al(u_\mu)=(u_\mu\oti1)w_\mu^o
\quad
\mbox{for }
\mu\in P.
\]

Put $\cP:=W^*(u_\mu\mid\mu\in P)\vee\cQ$.
Let us introduce a unitary map
$U\col L^2(G_q)\ra L^2(\cP)$
such that
$U(v_\mu a 1_h)=u_\mu \pi(a) 1_\ovph$
for $\mu\in P$ and $a\in\lTG$.
Then we have
$Uv_\mu U^*=u_\mu$
and
$UaU^*=\pi(a)$ for $\mu\in P$ and $a\in\lTG$.
The map $\pi$ extends to a map,
which we also denote by $\pi$,
from $\lG$ into $\cM$.
The $G_q$-equivariance of $\pi$
is verified as
\begin{align*}
\bal(\Ad U(v_\mu))
&=
\bal(u_\mu)
=
(u_\mu\oti1)w_{\mu}^o
\\
&=
(u_\mu\oti1)(\pi\oti\id)(w_{\mu})
=
(\Ad U\oti\id)
((v_{\mu}\oti1)w_{\mu})
\\
&=
(\Ad U\oti\id)(\de(v_\mu)).
\end{align*}
\end{proof}

\begin{rem}
It turns out from the previous lemma
that $\al$ is semidual.
Hence there exists an action $\si$
of $\widehat{G}_q$
on $\cN=B(\ell^2)\oti\cM^\al$
such that
$\ovl{\cM}=\cN\rti_\si\widehat{G}_q$.
\end{rem}

Recall the restriction of an action
by a quantum subgroup
(see Section \ref{subsect:quantum subgrp}).
In the following lemma,
we will show that
the minimal action $\be$
actually comes from the restriction of $\al$
by the maximal torus $T$
though it seems not so clear at first.

Let $\al_T$ be the restriction of $\al$ on $T$.
We denote by $\al_t$ the restriction of $\al$ on $t\in T$,
that is, $\al_t:=(\id\oti\ev_t)\circ\al_T$ for $t\in T$.
Let $w_0t$ be the element
satisfying
$\langle w_0t,\mu\rangle=\langle t,w_0\mu\rangle$
for all $\mu\in P$.

\begin{lem}
\label{lem:restriction}
The minimal action $\be_t$ of $T$ on $\cR$
is given by $\al_{w_0 t}$ on $\cR$.
\end{lem}
\begin{proof}
To see this,
we may assume that $\cM^\al$ is infinite.
Then $\cR$ is generated by $\cM^\al$
and $u_\la$'s as before.

By the above equivariant embedding $\pi$,
$\al_T$ on $\{u_\la\}_\la''\vee\cQ$
is conjugate to the right torus action $\ga^R$
on $\lG$,
where $\ga_t^R:=(\id\oti\ev_t\circ\, r_T)\circ\de$
for $t\in T$.
Using $\pi(v_\la)=u_\la$ and
the polar decomposition of $C_{\La,w_0\La}^\La$
with $\La\in P_+$,
we have
$\al_t(u_\la)=\pi(\ga_t^R(v_\la))
=\langle t,w_0\la\rangle u_\la=\be_{w_0t}(u_\la)$.
\end{proof}

\begin{rem}
\label{rem:unitary-impl}
Let $x\in\R^n$
and $y:=A^{-1}x$, where $A$ denotes the Cartan matrix.
Then we put $t=(t_j)_j$
with $t_j=q_j^{iy_j}$ for $j=1,\dots,n$,
and we get
$(w_0t,\nu)=\prod_j q^{i(w_0\om_j,\nu)x_j}$.
By the commutation relation in the proof of Theorem \ref{thm:factor},
we obtain
\[
\ga_{w_0t}^R=\Ad |a_{\om_1}|^{ix_1}\cdots|a_{\om_n}|^{ix_n}
\quad\mbox{on }
\lTG.
\]
This shows the right action $\ga^R$ on $\lTG$
is implemented by a unitary representation.
\end{rem}

\begin{lem}
The map
\[
\Xi\col
(\ovl{\cM}^{\bal}\oti\C)\vee
W^*(u_\mu\oti v_\mu\mid\mu\in P)\vee
(\C\oti\lTG)
\ra
\ovl{\cM}
\]
with $\Xi((a\oti1)(u_\mu\oti v_\mu)(1\oti b))
=au_\mu\pi(b)$
for $a\in\ovl{\cM}^{\bal}$, $\la\in P$
and $b\in\cQ$
is a well-defined $G_q$-equivariant isomorphism.
\end{lem}
\begin{proof}
Let
$\cL:=(\ovl{\cM}^{\bal}\oti\C)\vee
W^*(u_\mu\oti v_\mu\mid\la\in P)\vee
(\C\oti\lTG)$.
Then $\cL\subs \ovl{\cR}\oti \lG$.

\begin{clm}
The following map $U\col L^2(\cL)\ra L^2(\ovl{\cM})$
is a well-defined unitary:
\[
U((a\oti1)(u_\mu\oti v_\mu)(1\oti b)
(1_{\ovl{\vph}}\oti1_h))
:=au_\mu\pi(b)1_{\ovl{\vph}}
\]
for $a\in\ovl{\cM}^{\bal}$,
$\mu\in P$
and $b\in \lTG$,
where $\pi$ is the one defined in the previous lemma.
\end{clm}
\begin{proof}[Proof of Claim]
Recall that $\ovl{\cM}\cong \ovl{\cR}\oti\cQ$
and $\ovl{\vph}$ is splitted to
$\ovl{\vph}|_{\ovl{\cR}}\oti\vph_{\cQ}$.
Then the well-definedness
follows from $\ovl{\vph}(au_\mu)=0$
for $a\in\ovl{\cM}^{\bal}$
and a non-zero $\mu\in P$.
\end{proof}
Using this map,
we obtain an isomorphism $\Xi\col\cL\ra \ovl{\cM}$
as in the statement.
We will check the $G_q$-equivariance.
Let $a\in\ovl{\cM}^{\bal}$,
$\mu\in P$ and $b\in \lTG$.
Then $\Xi(a)=a$ and $\Xi(1\oti b)=\pi(b)$.
Next,
\begin{align*}
(\Xi\oti\id_{\lG})
\left((\id_{\ovl{\cR}}\oti\de)(u_\mu\oti v_\mu)\right)
&=
(\Xi\oti\id_{\lG})
\left((u_\mu\oti v_\mu\oti1)
(1\oti w_\mu)
\right)
\\
&=
(u_\mu\oti1)(\pi\oti\id)(w_\mu)
\\
&=
(u_\mu\oti1)w_\mu^o
=\bal(u_\mu)
\\
&=
\bal(\Xi(u_\mu\oti v_\mu)).
\end{align*}
Therefore, $\Xi$ is $G_q$-equivariant.
\end{proof}

We will recall the notion of the induction of actions.

\begin{defn}
Let $\bH$ be a quantum subgroup of $\bG$
and $\Ga\col \cA\ra\cA\oti L^\infty(\bH)$
an action of $\bH$ on a von Neumann algebra $\cA$.
Let $\ga_\bH:=(r_\bH\oti\id)\circ\de$ be the left action
of $\bH$ on $L^\infty(\bG)$.
Set
\[
\Ind_\bH^{\bG}\cA
:=
\cA\oti_\bH L^\infty(\bG)
=
\{x\in\cA\oti L^\infty(\bG)
\mid (\Ga\oti\id)(x)=(\id\oti \ga_\bH)(x) \}.
\]
Then the restriction of $\id\oti\de$ on $\Ind_\bH^{\bG}\cA$,
which we will denote by $\Ind_\bH^{\bG}\Ga$,
gives an action of $\bG$,
and we will call it the \emph{induction} of $\Ga$ from $\bH$ to $\bG$.
\end{defn}

Note that the fixed point algebra of $\Ind_\bH^{\bG}\Ga$
is equal to $\cA^\Ga$.
Now we will prove the following main result of this paper.

\begin{thm}
\label{thm:induction}
A faithful product type action
of $G_q$ is induced from a minimal
action of $T$ on a type III factor.
Moreover, such minimal action is unique in the following sense:
If there exists a minimal action $\chi$ of $T$ on a factor $\cN$
such that $\Ind_T^{G_q}\cN$ is $G_q$-equivariantly
isomorphic to $\cM$,
then there exist a $*$-isomorphism $\zeta$
from $\cR$ onto $\cN$
and a topological group isomorphism
$f$ on $T$ such that
$\chi_t=\zeta\circ\be_{f(t)}\circ\zeta^{-1}$.
\end{thm}
\begin{proof}
We let $\cA:=W^*(u_\mu\mid\la\in P)\subs\ovl{\cR}$.
Since $\be_t(u_\mu)=\langle t,\mu\rangle u_\mu$
and $\ga_t(v_\mu)=\langle t,\mu\rangle v_\mu$,
we have
\[
\cA\oti_T Z(\lTG)=W^*(u_\mu\oti v_\mu\mid\la\in P).
\]
Therefore,
\begin{align*}
B(\ell^2)\oti\Ind_T^{G_q}\cR
&=\ovl{\cR}\oti_T\lG
=(\ovl{\cM}^{\bal}\vee \cA)\oti_T (Z(\lG)\vee\lTG)
\\
&=
(\ovl{\cM}^{\bal}\oti\C)\vee
W^*(u_\mu\oti v_\mu\mid\la\in P)\vee
(\C\oti\lTG),
\end{align*}
which is isomorphic to $B(\ell^2)\oti\cM$
through $\Xi$,
the map constructed in the previous lemma.
By definition,
$\Xi$ maps the fixed point algebra
$(\ovl{\cR}\oti_T\lG)^{G_q}=\ovl{\cR}^{\ovl{\be}}
=\ovl{\cM}^{\bal}$
onto $\ovl{\cM}^{\bal}$ identically.
Thus we can remove the contribution of $B(\ell^2)$.

Next suppose that
we have a minimal action $\chi$ of $T$ on a factor $\cN$
such that $\cP:=\Ind_T^{G_q}\cN$ is $G_q$-equivariantly
isomorphic to $\cM$.
Then the inclusion
$\cN^\chi\subs (\C\oti\lTG)'\cap \cP$
is isomorphic to
$\cM^\al\subs\cR$.
Thus there exist a $*$-isomorphism $\zeta$
from $\cR$ onto $\cN$
and a topological group isomorphism
$f$ on $T$ such that
$\chi_t=\zeta\circ\be_{f(t)}\circ\zeta^{-1}$
since
every automorphism $\ps$ on $\cR$ which fixes $\cM^\al$
is of the form
$\be_t$ for some $t\in T$.
Indeed, $u_\la^*\ps(u_\la)$ commutes with $\cM^\al$,
and it is a scalar
(see \cite[p.131]{AHKT} for a more general situation).
\end{proof}

\subsection{Classification of product type actions}
As mentioned in Lemma \ref{lem:restriction},
the minimal action $\be$ on $\cR$ comes from the restriction of $\al$
on $T$, which we denote by $\al_T$ as usual.
Let $\al_t:=(\id\oti\ev_t)\circ\al_T$ for $t\in T$.
Readers are referred to \cite{MT,Ocn}
for the notion of \emph{conjugacy} and \emph{cocycle conjugacy}.
We will say that a (quantum) group action is \emph{stable}
when every cocycle is a coboundary.

Recall that $\al_T$ on $\cQ$ is implemented by
a unitary representation (see Remark \ref{rem:unitary-impl}).
Then we have
\begin{align*}
\al_{w_0t}
&\approx \al_{w_0t}|_\cR\oti \al_{w_0t}|_\cQ
= \be_t\oti\al_{w_0t}|_\cQ
\\
&\sim \be_t\oti\id|_\cQ
\\
&\sim \be_t
\quad
\mbox{for }t\in T,
\end{align*}
where we have used the infiniteness of $\cR$
at the last cocycle conjugacy.
The notations $\approx$ and $\sim$ denote
the conjugacy and the cocycle conjugacy, respectively.
We will summarize this observation in the following
(cf. Lemma \ref{lem:restriction}).
For the notion of \emph{invariant approximate innerness},
readers are referred to \cite[Definition 4.5, Lemma 4.7]{MT}.

\begin{thm}
\label{thm:beta-prod}
The minimal action $\be_t$ of the maximal torus
$T$ on $\cR$ is cocycle conjugate to
$\al_{w_0t}$.
In particular,
$\be$ is invariantly approximately inner.
\end{thm}

Note that $\cM$ is the completion of the infinite tensor product
of $B(H)$ by a product state,
$\cM$ is of type III$_\la$ with $0<\la\leq1$.
To compute the type of $\cM^\al$,
the following result is useful.
Note that $\cM^\al$ is not of type I as remarked
in the proof of Lemma \ref{lem:Hilb-embed}.

\begin{cor}
\label{cor:beta-fixed-type}
The following statements hold:
\begin{enumerate}
\item
The fixed point algebra $\cM^{\al_T}$ is not of type III$_0$;

\item 
If $\cM^{\al_T}$ is of type III$_\la$ with $0<\la\leq1$,
then so is $\cM^\al$.
In this case, $\al$ is stable;

\item
If $\cM^{\al_T}$ is of type II,
then so is $\cM^\al$.
\end{enumerate}
\end{cor}
\begin{proof}
(1).
It is clear that the canonical action
of the infinite symmetric group $\mathfrak{S}_\infty$
is commuting not only $\al_T$ but $\si^\vph$,
where $\vph$ is the product state with respect to $\ph$.
Therefore,
$(\cM^{\al_T})_\vph'\cap\cM^{\al_T}=\C$,
and $\Ga(\si^\vph|_{\cM^{\al_T}})=\Sp(\si^\vph|_{\cM^{\al_T}})$.
This shows that $\cM^{\al_T}$ is not of type III$_0$.

(2).
By \cite[Proposition 5.2 (4)]{Iz-can},
$\al_T$ is stable.
This implies that $\al_t$ is conjugate to $\be_{w_0t}$,
and $\cM^{\al_T}\cong \cR^\be=\cM^\al$.
The stability of $\al$ is shown
by using 2$\times$2-matrix trick.

(2).
If $\cM^\al=\cR^\be$ were of type III,
then so would $\cM^{\al_T}$
since there exists a normal conditional expectation
from $\cM^{\al_T}$ onto $\cM^\al$.
This is a contradiction.
\end{proof}

Theorem \ref{thm:beta-prod}
enables us to classify
some product type actions of $G_q$.

\begin{cor}
A product type action $\al$
is unique up to conjugacy
if $\cM^\al$ is of type III$_1$.
More precisely,
such $\al$ is conjugate to
$\Ind_T^{G_q}(\id_{\cR_\infty}\oti\, m)$,
where $\cR_\infty$ denotes the injective 
type III$_1$ factor and $m$ the minimal
action of $T$ on the type II$_1$
injective factor $\cR_0$.
\end{cor}
\begin{proof}
Let $\be$ be the associated minimal action on $\cR$.
Then $\cR^\be=\cM^\al$ is of type III$_1$.
It follows
that $\be$ is a dual action of an outer action
$\th^{-1}$ on $\cR^\be$.
Then $\th_\mu$ for each $\mu\in \widehat{T}$
is approximately inner
since $\Aut(\cR_\infty)=\oInt(\cR_\infty)$
\cite[Theorem 1]{KST}.
By \cite[Theorem 4.11]{MT},
$\th$ has the Rohlin property, that is,
the central freeness.
Thus $\th$ is unique up to cocycle conjugacy
\cite[Theorem 1.4, p.7]{Ocn}.
This implies the uniqueness of $\be$
up to conjugacy.
\end{proof}

\begin{exam}
We will construct a model of a product type action
whose fixed point algebra is of type III$_1$.
As a result,
it turns out that
$\Ind_T^{G_q}(\id_{\cR_\infty}\oti\, m)$
is indeed of product type.

Take an $n$-dimensional unitary representation
$v$ of $G_q$
such that the matrix elements
$v_{ij}$ generate $C(G_q)$.
Then we set the $(n+3)$-dimensional representation
$w:=1^{\oplus3}\oplus v$.
Let $\la,\mu>0$ such that $\la/\mu\nin\Q$.
We introduce $\Ad w$-invariant state $\ph$
defined by
the normalisation of $\Tr_k$,
where
$k$ denotes the diagonal matrix
$\diag(1,\la,\mu,F_v)$.

By Corollary \ref{cor:beta-fixed-type},
it suffices to show that
$\cM^{\al_T}$ is of type III$_1$.
It follows from the proof of Corollary \ref{cor:beta-fixed-type}
that $\Ga(\si^\vph|_{\cM^{\al_T}})=\Sp(\si^\vph|_{\cM^{\al_T}})$.
By construction of $w$,
it turns out that
$\log\la,\log\mu\in\Sp(\si^\vph|_{\cM^{\al_T}})$.
Thus $\cM^{\al_T}$ is of type III$_1$.
\end{exam}

When the fixed point algebra is of another type,
it seems that the general classification
is complicated.
So, let us treat $SU_q(2)$ in what follows.
Our main ingredient is the complete invariant
treated in \cite[Theorem 6.28]{MT}.
Note that two actions of the torus $\R/2\pi\Z$
are cocycle conjugate if and only if
so are they as $\R$-actions.

We now suppose that $\al$ is a product type action of $SU_q(2)$
and $v$ a finite dimensional representation.
To compute the invariant,
we give a parametrization of $v$ and $\ph$ as follows.
The irreducible representations of $SU_q(2)$
are parametrized by $\Z_+\om_1$,
or equivalently, the half spins $(1/2)\Z_+$.
Let us decompose $v$
into the direct sum of irreducible representations
as follows:
\[
v
=
\bigoplus_{\nu\in(1/2)\Z_+}
\bigoplus_{k=1}^{m_\nu}
C^\nu,
\]
where $m_\nu$ denotes the multiplicity of
$C^\nu$ in $v$.
Under identification of $T=\R/2\pi\Z$,
we have
\[
v_t
=
\bigoplus_{\nu\in(1/2)\Z_+}
\bigoplus_{k=1}^{m_\nu}
\diag(e^{2\nu it},e^{(2\nu-2)it},\dots,e^{-2\nu it})
\quad
\mbox{for }
t\in\R.
\]

Changing the orthonormal base of each intertwiner space
if necessary,
we may and do assume that $\ph$
is the normalization of $\Tr_{k_\ph}$,
where $k_\ph$ is defined as
\[
k_\ph=
\bigoplus_{\nu\in(1/2)\Z_+}
\bigoplus_{k=1}^{m_\nu}
c_k^\nu
\diag
(q^{2\nu},q^{2\nu-2},\dots,q^{-2\nu}),
\quad
\mbox{for some }
c_k^\nu>0.
\]
From the faithfulness of $\al$,
$v$ has at least one non-integer-spin representation
and at least one integer.
Thus we may assume that $c_k^\nu=1$ for a fixed even $\nu$
and $k$.
Note that the density matrix
$\diag(q^{2\nu},q^{2\nu-2},\dots,q^{-2\nu})$
contains $1$ as its spectrum for any integer-spin $\nu$.

Then the invariant $G_{{\bm\la},{\bm\mu}}$
stated in \cite[Theorem 6.28]{MT}
is computed as follows:
\[
G_{\al_T}
:=
\langle
(\log (c_k^\nu q^{\ell}),\ell)
\mid
\ell=2\nu,2\nu-2,\dots,-2\nu,
\
k=1,\dots,m_\nu,
\
\nu\in(1/2)\Z_+
\rangle,
\]
which is a closed subgroup of $\R^2$.
Since there exists $\nu\in(1/2)+\Z_+$ with $m_\nu>0$,
$G_{\al_T}$ can be written as the following form:
\begin{equation}
\label{eq:invariant-G}
G_{\al_T}
=
\langle
(\log c_k^{\nu_e},0),
(2\log c_k^{\nu_o},0),
(\log (c_k^{\nu_o}q),1)
\mid
k,
\
\nu_e\in\Z_+,
\
\nu_o\in 1/2+\Z_+
\rangle
\end{equation}

\begin{thm}
If $G_q=SU_q(2)$, and $\cM^\al$ is of type II,
then $\cM^\al$ and $\cM$ must be of type II$_1$ and III$_q$,
respectively.
Moreover,
$\al$ is conjugate to the induction
of the torus action $\si_{t/\log q}^{\vph_q}$,
where $\vph_q$ denotes the Powers state
on the Powers factor $\cR_q$ of type III$_q$.
\end{thm}
\begin{proof}
Let $\tr$ be the tracial weight on $\cR^\be$.
Then by \cite[Proposition 5.2 (5)]{Iz-can},
$\{\si_t^{\ta\circ E_\al|_{\cR}}\}_{t\in\R}$ is contained
in $\{\be_t\}_{t\in \R/2\pi\Z}$.
In particular,
$\si^{\ta\circ E_\al|_{\cR}}$ is periodic,
and $\cR$ is of type III$_\la$ for some $0<\la<1$.
We will show that $\la$ must be equal to $q$.

By \cite[Proposition 6.34]{MT},
$\be_t$ is cocycle conjugate to $\si_{t/\log\la}^{\ps_\la}$
or $\si_{-t/\log\la }^{\ps_\la}$,
where $\ps_\la$ denotes the Powers state
on the Powers factor $\cR_\la$.
From Theorem \ref{thm:beta-prod},
we have
$\al_t\sim\be_{-t}\sim \si_{\mp t/\log\la }^{\ps_\la}$,
and their invariants introduced in \cite[Section 6.5]{MT} coincide.
The invariant of $\si_{\mp t/\log\la }^{\ps_\la}$
equals
$G_\la=\Z(\log\la,\mp 1)$.
It follows immediately from (\ref{eq:invariant-G})
that $c_k^\nu=1$ for all $\nu$ and $k$,
and $\la=q$.
Hence we have $\be_t\sim\si_{-t/\log\la }^{\ps_\la}$
and $\al_t=\si_{t/\log q}^{\vph}$ for $t\in\R$.
So, $\cM^{\al_T}=\cM_\vph$ is of type II$_1$.

We will show that $\be_t$ is in fact conjugate to
$\si_{-t/\log q}^{\ps_q}$.
Employing Lemma \ref{lem:restriction},
we have $\be_t=\al_{-t}=\si_{-t/\log q}^{\vph}=\si_{-t/\log q}^{\vph|_\cR}$
on $\cR$.
Note that $\cR\cong\cR_q$ and $\vph$ and $\ps_q$ are periodic states.
Then by adjusting a Connes--Takesaki module,
there exists an isomorphism $\zeta\col\cR\ra\cR_q$ such that
$\vph|_\cR=\ps_q\circ\zeta$.
Thus $\be_t\approx\si_{-t/\log q}^{\ps_q}$.
It is not so difficult to show that
the induction of an action is stable
with respect to an automorphism of $T$,
and
we have
$\al\approx\Ind_T^{G_q}\si_{-t/\log q}^{\ps_q}
\approx\Ind_T^{G_q}\si_{t/\log q}^{\ps_q}$.
\end{proof}

\begin{exam}
Let $v$ be the direct sum of the spin-0 and 1/2 irreducible representations.
Namely, a unitary $v$ has the following form:
\[
v=
\begin{pmatrix}
1&0&0\\
0&x&u\\
0&v&y
\end{pmatrix}
\in
M_3(\C)\oti C(SU_q(2))
=
M_3(C(SU_q(2))),
\]
where $x,u,v$ and $y$ are the canonical generators
of $C(SU_q(2))$ as a C$^*$-algebra
(see \cite{MM} or Section \ref{subsect:class-irr-Soibelman}).

Now we set the following density matrix:
\[
k_\ph
=
\begin{pmatrix}
1&0&0\\
0&q&0\\
0&0&q^{-1}
\end{pmatrix}
.
\]
where $\Tr$ denotes the canonical non-normalized
trace of $M_3(\C)$.
Let $\vph$ be the product state of $\ph$ as usual.
Then $\al_t=\si_{-t/\log q}^\vph$ for $t\in\R/2\pi\Z$,
and $\cM^{\al_T}$ is of type II$_1$.
Thus so is $\cM^\al$.
\end{exam}

The remaining case is when $\cM^\al$ is of type III$_\la$
with $0<\la<1$.
The infiniteness of $\cR^\be$ implies that
the crossed product decomposition of $\cR$,
that is, $\cR=\cR^\be\rti_\th\Z$.
Recall that $\be_t=\hat{\th}_{e^{-it}}$ for $t\in\R/2\pi\Z$.
Since $\be$ is invariantly approximately inner,
$\th$ is centrally free.

Let $\mo(\th)$ be the Connes--Takesaki module of $\th$ \cite{CT}.
Identifying the flow space of $\cR^\be$ with $(\la,1]=\R_{>0}/\la^\Z$,
we may assume that $\la\leq\mo(\th)<1$.
Let $\mu:=\mo(\th)$.
Thanks to the classification of $\Z$-actions,
(see \cite[Theorem 1, Corollary 6, p.385]{Co}
or \cite[Theorem 1.13, p.311]{Ta-III}),
$\th$ is cocycle conjugate to $\id_{\cR_\la}\oti \th^\mu$,
where $\th^\mu$ denotes the automorphism
on the injective type II$_\infty$ factor $\cR_{0,1}$
with $\tr\circ\th^\mu=\mu\tr$.
Thus $\cR\cong \cR_\la\oti\cR_\mu$
and
$\be_t\approx\id_{\cR_\la}\oti\si_{t/\log\mu}^{\vph_\mu}$.
So, the invariant of $\be$ is computed as follows:
\begin{equation}
\label{eq:Glamu}
G_{\la,\mu}
:=
\Z(\log\la,0)+\Z(\log\mu,1).
\end{equation}

Note that we can replace $\mu$ with $\la\mu$,
that is,
$G_{\la,\la\mu}=G_{\la,\mu}$.
This shows that
$\id_{\cR_\la}\oti\si_{t/\log\mu}^{\vph_\mu}$
is (cocycle) conjugate to
$\id_{\cR_\la}\oti\si_{t/\log(\la\mu)}^{\vph_{\la\mu}}$.

\begin{thm}
If $G_q=SU_q(2)$, and $\cM^\al$ is of type III$_\la$
with $0<\la<1$,
then $\mo(\th)=q$ or $\la^{1/2}q$ in $\R_{>0}/\la^\Z$.
In each case,
$\al$ is unique up to conjugacy.
\end{thm}
\begin{proof}
In this case,
we have $G_{\al_T}=G_{\la,\mu}$.
By (\ref{eq:invariant-G}) and (\ref{eq:Glamu}),
we $c_k^{\nu_o}\in\la^{m_k}$ with $m_k\in(1/2)\Z_+$
and
$\mu\in q c_k^{\nu_o} \la^\Z$.
Hence $\mu=q\la^n$ for some $n\in(1/2)\Z_+$,
and $\mo(\th)=q$ or $\la^{1/2}q$ in $\R_{>0}/\la^\Z$.
\end{proof}

\begin{exam}
Let $0<\la<1$ and $\vep\in\{0,1/2\}$.
Set $v$ and $k_\ph$ as follows:
\[
v
=
\begin{pmatrix}
1&0&0&0\\
0&1&0&0\\
0&0&x&u\\
0&0&v&y
\end{pmatrix}
,
\quad
k_\ph
=
\begin{pmatrix}
1&0&0&0\\
0&\la&0&0\\
0&0&\la^\vep q&0\\
0&0&0&\la^\vep q^{-1}
\end{pmatrix}
.
\]
Then we can see
$G_{\al_T}=G_{\la,\la^\vep q}$ by direct calculation.

So, if $\mu=\la^k q<1$ with $k$ a half integer,
then $\Ind_{T}^{G_q}(\id_{\cR_\la}\oti \si_{t/\log\mu}^{\vph_\mu})$
is of product type and falls into two categories.
\end{exam}

The following result is a direct consequence of
the previous theorem.

\begin{cor}
Let $G_q=SU_q(2)$ and $0<\la<1$.
Suppose that $\mu$ satisfies $0<\mu<1$
and $\mu/q\nin (\la^{1/2})^{\Z_+}$.
Then the induced action
$\Ind_{T}^{G_q}(\id_{\cR_\la}\oti \si_{t/\log\mu}^{\vph_\mu})$
is not of product type.
In particular,
for any $0<\la<1$,
there exist uncountably many, non-product type,
mutually non-cocycle conjugate actions
of $SU_q(2)$ on the injective type III$_1$ factor
with fixed point factor of type III$_\la$.
\end{cor}

\section{Related problems}
Let $\bG$ be a compact quantum group
and $\bK$ the maximal quantum subgroup
of Kac type introduced in
\cite[App. A]{Sol} and \cite[Definition 4.6]{T-Poisson}.
We would like to generalize
Dijkhuizen-Stokman's result stated
in Section \ref{subsect:irreducible}.

\begin{prob}
\label{prob:irred-maximal}
Does the following equality hold?
\[
\Irr(C(\bK\backslash\bG))
=
\{
\pi|_{C(\bK\backslash\bG)}
\mid
\pi\in\Irr(C(\bG))
\},
\]
where $\Irr(A)$ denotes
the equivalence classes
of irreducible representation
of a ${\rm C}^*$-algebra $A$.
\end{prob}

\begin{prob}
Is the counit a unique character
on $C(\bK\backslash\bG)$?
\end{prob}

We will remark on this problem.
Let $\Ga$ be the set of characters on $C(\bG)$.
Then it is probably well-known for experts
that $\Ga$ is a compact group
that is regarded as a quantum subgroup
of $\bG$.
The maximality of $\bK$ implies that
$C(\bK\backslash\bG)\subs C(\Ga\backslash\bG)$.
In particular, the restriction
of every element of $\Ga$
on $C(\bK\backslash\bG)$ gives a counit.
Thus if Problem \ref{prob:irred-maximal} is solved,
this problem holds.

\begin{prob}
$\Aut_{\bG}(C(\bK\backslash\bG))=\{\id\}$?
\end{prob}

If $\bG$ is a compact group,
then $\bK=\bG$.
So these problems are trivial.
We will explain
why the last problem seems plausible.
Let $G$ be a compact group and $H$ a closed subgroup
of $G$.
Then $\Aut_G(C(H\backslash G))$ is isomorphic to
$N_G(H)/H$, where $N_G(H)$ denotes the normalizer group
of $H$.
If there exists a non-trivial $g\in N_G(H)$,
then $H$ and $g$ generate a closed subgroup
larger than $H$.
Hence the maximality of $\bK$
would imply the triviality of
$\Aut_{\bG}(C(\bK\backslash\bG))$.

In the last section,
in order to show that
a faithful product type action of $G_q$ is induced
from a minimal action of the maximal torus $T$,
we have exploited
the representation theory of $G_q$ and $C(G_q)$
to a full.
We would like to obtain this
in a more conceptual way.

\begin{prob}
Let $\bG$ be a co-amenable compact quantum group
with commutative fusion rules
and $\bK$ the maximal quantum subgroup of Kac type.
Then is any faithful product type action of $\bG$
induced from a minimal action of $\bK$?
\end{prob}

\end{document}